\RequirePackage[table]{xcolor}

\documentclass[a4paper,11pt]{amsart}

\usepackage{bbm}
\usepackage{amsfonts}
\usepackage{amsmath}
\usepackage{amssymb}
\usepackage{amsthm}
\usepackage{color}
\usepackage{caption}

\usepackage[shortlabels]{enumitem}

\usepackage[paper=a4paper,left=20mm,right=20mm,top=25mm,bottom=30mm]{geometry}
\usepackage[colorlinks=true, linktocpage=true]{hyperref}
\usepackage{mathrsfs}
\usepackage{mathtools}
\usepackage[numbers]{natbib}
\usepackage{newpxtext}
\usepackage{soul}
\usepackage{url}
\usepackage[all,cmtip]{xy}
\usepackage{tikz,pgfplots}
\usetikzlibrary{backgrounds}
\usepackage{tkz-base}
\usepackage{tkz-euclide}
\usepackage{eucal}
\usepackage{verbatim}
\usepackage{mathtools}
\usepackage{palatino}
\usepackage[capitalise]{cleveref} 
\crefformat{equation}{(#2#1#3)}
\crefmultiformat{equation}{(#2#1#3}{ \& #2#1#3)}{, #2#1#3}{ \& #2#1#3)}
\crefname{construction}{Construction}{Constructions}
\crefrangeformat{equation}{(#3#1#4--#5#2#6)}

\newcommand{\bmk}{\mathbbm{k}}

\newcommand{\cB}{\mathcal{B}}

\newcommand{\cC}{\mathcal{C}}

\newcommand{\cE}{\mathcal{E}}
\newcommand{\cF}{\mathcal{F}}

\newcommand{\cO}{\mathcal{O}}

\newcommand{\cR}{\mathcal{R}}
\newcommand{\cS}{\mathcal{S}}
\newcommand{\cT}{\mathcal{T}}
\newcommand{\cU}{\mathcal{U}}

\newcommand{\fM}{\mathfrak{M}}

\newcommand{\fu}{\mathfrak{u}}
\newcommand{\fv}{\mathfrak{v}}

\newcommand{\lra}{\longrightarrow}

\newcommand{\NN}{\mathbb{N}}

\newcommand{\QQ}{\mathbb{Q}}
\newcommand{\RR}{\mathbb{R}}

\newcommand{\ZZ}{\mathbb{Z}}

\DeclareMathOperator{\add}{add}

\DeclareMathOperator{\Aut}{Aut}
\DeclareMathOperator{\bd}{\mathbf{d}}
\DeclareMathOperator{\be}{\mathbf{e}}
\DeclareMathOperator{\bff}{\mathbf{f}}

\DeclareMathOperator{\Coh}{Coh}

\DeclareMathOperator{\dimu}{\underline{dim}}

\DeclareMathOperator{\End}{End}

\DeclareMathOperator{\esp}{esp}

\DeclareMathOperator{\Ext}{Ext}

\DeclareMathOperator{\GL}{GL}
\DeclareMathOperator{\Gr}{Gr}

\DeclareMathOperator{\hlra}{\hookrightarrow}
\DeclareMathOperator{\Hom}{Hom}

\DeclareMathOperator{\id}{id}
\DeclareMathOperator{\im}{im}

\DeclareMathOperator{\inj}{inj}

\DeclareMathOperator{\modd}{mod}

\DeclareMathOperator{\proj}{proj}

\DeclareMathOperator{\reg}{reg}

\DeclareMathOperator{\rep}{rep}

\DeclareMathOperator{\rk}{rk}

\DeclareMathOperator{\StVect}{StVect}

\DeclareMathOperator{\Vect}{Vect}
\renewcommand{\to}{\lra}
\let\emptyset\varnothing

\newtheorem{proposition}{Proposition}[section]
\newtheorem{Theorem}[proposition]{Theorem}
\newtheorem{Lemma}[proposition]{Lemma}

\newtheorem{corollary}[proposition]{Corollary}

\newtheorem{TheoremA}{Theorem}

\theoremstyle{remark}
\newtheorem{Remarks}[proposition]{Remark}
\newtheorem{Remark}[proposition]{Remark}

\theoremstyle{definition}
\newtheorem*{Definition}{Definition}

\newtheorem*{example}{Example}
\newtheorem*{examples}{Example}

\pgfplotsset{compat=1.18} 

\newlist{enumerate2}{enumerate}{10}
\setlist[enumerate2]{label={({\alph*})}}
\setlist[enumerate]{label={({\arabic*})}}

\hypersetup{
  linkcolor  = blue, %Color Referenz
  citecolor  = blue, % Color Cite
  urlcolor   = blue, % Color Url
  filecolor = red,  
}

\title{Shift orbits for elementary representations of Kronecker quivers}

\begin{document}

\rmfamily

%%%%%%%%%%%%%%%%%%%%%%%%%%%%%%%%%%%%%%%%%%%%%%%%%%%%%%%%%%%%%%%%%%

\address{Daniel Bissinger, Christian-Albrechts-Universit\"at zu Kiel, Heinrich-Hecht-Platz 6, 24118 Kiel, Germany}
\email{bissinger@math.uni-kiel.de}
\author{Daniel Bissinger}

\maketitle

%%%%%%%%%%%%%%%%%%%%%%%% Subjects classification %%%%%%%%%%%%%%%%%%%
\makeatletter
\newcommand\blfootnote[1]{%
  \begingroup
  \renewcommand\thefootnote{}\footnote{#1}%
  \addtocounter{footnote}{-1}%
  \endgroup
}

\makeatother

\blfootnote{\textup{2020} {\it Mathematics Subject Classification}: 16G20}

\maketitle

%%%%%%%%%%%%%%%%%%%%%%%%%% Abstract %%%%%%%%%%%%%%%%%%%%%%%%%%%%%%%%%%%%%%%%%%%%%%%%
\begin{abstract} 
    Let $r \in \NN_{\geq 3}$. We denote by $K_r$ the wild $r$-Kronecker quiver with $r$ arrows $\gamma_i \colon 1 \lra 2$ and consider the action of the group $G_r \subseteq \Aut(\ZZ^2)$ generated by $\delta \colon \ZZ^2 \to \ZZ^2, (x,y) \mapsto (y,x)$ and $\sigma_{r} \colon \ZZ^2 \to \ZZ^2, (x,y) \mapsto (rx-y,x)$ on the set of regular dimension vectors \[\cR = \{ (x,y) \in \NN^2 \mid x^2 + y^2 - rxy < 1\}.\]
    A fundamental domain of this action is given by $\cF_r \coloneqq \{ (x,y) \in \NN^2 \mid  \frac{2}{r} x \leq y \leq x  \}$.
    We show that $(x,y) \in \cF_r$ is the dimension vector of an elementary representation if and only if 
    \[y \leq \min \{  \lfloor \frac{x}{r} \rfloor+\frac{x}{\lfloor \frac{x}{r} \rfloor}  -r, \lceil \frac{x}{r} \rceil   -\frac{x}{\lceil \frac{x}{r} \rceil} +r,r-1\},\]
    where we interpret $\lfloor \frac{x}{r} \rfloor+\frac{x}{\lfloor \frac{x}{r} \rfloor}  -r$ as $\infty$ for $1 \leq x < r$.
       In this case we also identify the set of elementary representations as a dense open subset of the irreducible variety of representations with dimension vector $(x,y)$. A complete combinatorial description of elementary representations for $r = 3$ has been given by Ringel. We show that such a compact description is out of reach when we consider $r \geq 4$, altough the representation theory of $K_3$ is as difficult as the representation theory of $K_r$ for $r \geq 4$. 
    \end{abstract}

%%%%%%%%%%%%%%%%%%%%%%%%%% Introduction %%%%%%%%%%%%%%%%%%%%%%%%%%%%%%%%%%%%%%%%%%%%%%%%

\section*{Introduction}

Let $\bmk$ be an algebraically closed field of arbitrary characteristic and $Q$ be a finite, connected and wild quiver without oriented cycles. It is well known that the theory of finite dimensional representations over $Q$ is undecidable (see \cite[4.4]{Ben91}, \cite{Pre96}), which makes a full classification of the category $\rep(Q)$ of finite dimensional representations a hopeless task.

The indecomposable representations of $Q$ fall into three classes: There are countable many (isomorphism classes of) so-called preinjective and preprojective indecomposable representations that are well-understood. All other indecomposable representations are called regular. A (not necessarily indecomposable) representation is called regular if all of its indecomposable direct summands are regular and we denote by $\reg(Q) \subseteq \rep(Q)$ the full subcategory containing all regular representations. This subcategory contains the large majority of representations and is responsible for the wild behaviour of the category $\rep(Q)$. 

Since regular representations are closed under extensions, there is a uniquely determined smallest class of regular representation $\cE \subseteq \reg(Q)$ closed under isomorphisms, whose extension-closure is $\reg(Q)$. In particular, every representation $M$ possesses a (in general not uniquely determined) finite filtration
\[ 0 = M_0 \subset M_1 \subset M_2 \subset \cdots \subset M_{l-1} \subset M_l = M\]
with $M_i/M_{i-1} \in \cE$ for all $i \in \{1,\ldots,l\}$. The representations in $\cE$ are called {\it elementary} and are the simple objects in the category of regular representations. The definition of elementary representations is due to Crawley-Boevey and is a natural generalization of quasi-simple representations living in regular tubes of tame hereditary algebras. 

Among other things, elementary representations are of interest because they are closely related (see \cite[3.1]{KL96}) to the graph of domination (see \cite[15.2, 15.3]{Ker96} for a precise definition), whose sinks are given by the extensively studied wild Kronecker algebras corresponding to generalized Kronecker quivers
\[ \xymatrix{
K_r =  & 1 \ar@/^1.4pc/^{\gamma_1}[r] \ar@/^/^{\gamma_2}[r]\ar@/_1pc/_{\gamma_r}[r] \ar@{}[r]|{\vdots} &  2 & , r \in \NN_{\geq 3}.
}
\]
Since a  representation in $\rep(Q)$ is elementary if and only if its Auslander-Reiten translate $\tau_{Q}(E)$ is elementary and the Coxeter transformation describes the $\tau_{Q}$-orbits on the level of the Grothendieck group, it is natural to consider Coxeter-orbits that belong to elementary representations.\\
These orbits have been studied systematically in \cite{Luk92,KL96} and it has been shown that there are only finitely many Coxeter-orbits of dimension vectors of elementary representations. The explicit number $e(Q) \in \NN$ of Coxeter-orbits of elementary representations is known in a few cases (see for example \cite[4.2.1]{Luk92}). But even for generalized Kronecker quivers this was unknown until recently: In \cite{Rin16} the equality $e(K_3) = 4$ has been proven.

In this article we tackle the general case and arrive at a criterion thats allows us to decide wether or not a dimension vector $(x,y)$ is the dimension vector of an elementary representation. In particular, we can decide wether a Coxeter-orbit belongs to the dimension vector of an elementary representation. As noted in \cite{Rin16}, it suffices to identify the elements in 
\[ \cF_r \coloneqq \{(x,y) \in \NN^2 \mid \frac{2}{r} x \leq y \leq x\} \]
that are the dimension vector of an elementary representation to obtain such a criterion. We follow this approach and arrive at:
\begin{TheoremA}
An element $(x,y) \in \cF_r$ is the dimension vector of an elementary representation if and only if
    \[(\ast) \quad y \leq \min \{  \lfloor \frac{x}{r} \rfloor+\frac{x}{\lfloor \frac{x}{r} \rfloor}  -r, \lceil \frac{x}{r} \rceil   -\frac{x}{\lceil \frac{x}{r} \rceil} +r,r-1\},\]
    where we interpret $\lfloor \frac{x}{r} \rfloor+\frac{x}{\lfloor \frac{x}{r} \rfloor}  -r$ as $\infty$ for $1 \leq x < r$.
\end{TheoremA}

In the following we outline the structure of this article and point out differences to \cite{Rin16} in the proof of Theorem A. A crucial step in \cite{Rin16}, to show that an elementary representation $E$ with dimension vector $\dimu E \in \cF_3$ has to satisfy $\dimu E \in  \{(1,1),(2,2)\}$, is an elegant application of the Projective Dimension Theorem (see \cite[I.7.2]{Har77}). The Projective Dimension Theorem is used to prove that every $K_3$-representation $M$ with dimension $(x,y)$ and and $2 \leq y \leq x+1$ has a subrepresentation with dimension vector $(1,2)$.

In the case $r \geq 4$ this geometric tool no longer yields strong enough restrictions on dimension vectors in $\cF_r$ that are the dimension vector associated to an elementary representation: For $r = 4$ the approach does not rule out the dimension vectors $(3,3) \in \cF_4$ altough it can not belong to an elementary representation by Theorem A.

Our approach rests on the consideration of the full subcategories $\rep_{\proj}(K_r,d) \subseteq \rep(K_r)$ ($1 \leq d < r$), of so-called {\it relative $d$-projective} Kronecker representations, each being equivalent to the category of Steiner bundles on the Grassmannian $\Gr_d(A_r)$ (see \cite[3.2.3]{BF24}), where $A_r = \bigoplus^r_{i=1} \bmk \gamma_i$ denotes the arrow space of the path algebra $\bmk K_r$. Restrictions on the minimal rank of non-trivial Steiner bundles, first established in \cite{AM15}  for fields of characteristic zero, and the connection between relative projective representations and representations with the so-called equal socle property (this definition originated from modular representation theory of finite group schemes) allow us to prove that an elementary representation $E$ with dimension vector $(x,y) \in \cF_r$ has to satisfy $(\ast)$.
For $(x,y) \in \cF_r$ satisfying $(\ast)$, we show that the elements $f \in \rep(K_r;\bmk^x,\bmk^x) \coloneqq \Hom_\bmk(\bmk^x,\bmk^y)^r$ with $(\bmk^x,\bmk^y,f) \in \rep(K_r)$ elementary form an open set $\cE(x,y)$ in the affine variety $\rep(K_r;\bmk^x,\bmk^x)$. We do so by showing that being relative $d$-projective and having the equal socle property is an open property. Moreover, we prove that this set has to be non-empty by dimension reasons, showing that the assumptions in Theorem A are sufficient.

An important tool in the proof of Theorem A is a new description by Reineke (see \cite[3.4]{Rei22}) concerning general subrepresentations of Kronecker representations. We characterize the category of representations with the equal $d$-socle property as those representations that do not have subpresentations with dimension vector $(1,a)$ for all $a \in \{0,\ldots,r-d\}$. This allows us to apply Reineke's Theorem and generalize Ringel's approach.

In the last section of this article we study the internal structure of elementary representations for $K_r$ with $r \geq 3$. For the tame Kronecker quiver $K_2$, the quasi-simple representations are well-known and completely determined in terms of their coefficient quiver, i.e. there exists a non-zero element $\alpha \in A_2$ such that the coefficient quiver has the following form:
\[
\xymatrix{
\bullet \ar[d]^{\alpha}  \\
\bullet.
}
\]
 For $r = 3$, the elementary representations with dimension vector in $\cE_3 = \{(1,1),(2,2)\}$ can also be described combinatorially in terms of the coefficient quiver (see \cite{Rin16}). More precisely: There exists a basis $\alpha,\beta,\gamma$ of the arrow space $A_3$ such that the coefficient quiver has one of the following two forms:
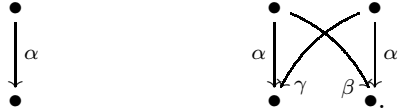
\[
\xymatrix{
\bullet \ar[d]^{\alpha} & & & \bullet \ar[d]_{\alpha} \ar@/^/[dr]_>>{\beta} & \bullet \ar[d]^{\alpha} \ar@/_/[dl]^>>{\gamma} \\
\bullet & & & \bullet & \bullet.
}
\]
Rephrasing this in the terms of the natural action of the general linear group $\GL(A_r)$ on $\rep(K_r)$, this just means that a representation $E \in \rep(K_3)$ with dimension vector in $\cF_3$ is elementary if and only if $M$ is isomorphic to an element in the $\GL(A_r)$-orbit of $E_1 \coloneqq (\bmk,\bmk,(\id_\bmk,0,0))$ or $E_2 \coloneqq (\bmk^2,\bmk^2,(\id_{\bmk^2}, \beta,\gamma))$ with $\beta(a,b) = (0,a)$ and $\gamma(a,b) = (b,0)$ for all $(a,b) \in \bmk^2$. Since the action of $\GL(A_r)$ on $\rep(K_r)$ commutes with the Auslander-Reiten translation $\tau_{K_r}$, we therefore can compute every elementary representation from $E_1$ and $E_2$. We show that the situation is quite different for $r \geq 4$.
\begin{TheoremA}
Let $r \geq 4$. Then there are infinitely many, pairwise non-isomorphic elementary representations with the same dimension vector $(x,y) \in \cF_r$ that all are in different $\GL(A_r)$-orbits.
\end{TheoremA}
It is well known that $\bmk K_s$ is wild algebra if and only $s \geq 3$. In particular, the representation theory of $K_3$ is as difficult as the representation theory of $K_r$ for $r \geq 4$. Moreover, in all cases, known to the author, proofs for $K_3$ can be easily generalized to $K_r$ by substituting $r$ for 3.
However, the above theorem tells us that the problem of classifying elementary representations gets much more difficult, when we consider $r \geq 4$ arrows.

\section*{Acknowledgement}
I would like to thank Rolf Farnsteiner for helpful comments and suggestions on an earlier version of this article, that helped to improve the exposition of this article. 
Furthermore I thank Claus Michael Ringel for sharing his insights on elementary Kronecker representations and answering questions regarding the first draft of \cite{Rin16}.

\setcounter{section}{0}
\section{Preliminaries}
Throughout $\bmk$ denotes an algebraically closed field of arbitrary characteristic and all $\bmk$-vector spaces under consideration are of finite dimension.

\subsection{Wild quivers} We assume that the reader is familiar with basic results on the representations theory of wild quivers. In the following, we will give a brief introduction, recalling the main definitions that we will use throughout this work.
For a well written survey on the subject, where all the details and proofs may be found, we refer to \cite{Ker96}.\\
Let $Q$ be a finite, connected and wild quiver without oriented cycles and vertex set $Q_0 = \{1,\ldots,n\}$. We denote by $\rep(Q)$ the category of finite dimensional representations over $Q$ and let $\bmk Q$ be the corresponding path algebra. The category $\rep(Q)$ and the category of finite dimensional (left) $\bmk Q$-modules $\modd \bmk Q$ are equivalent which allows us to switch freely between representations and modules.\\
Let $M \in \modd \bmk Q$. Then $\Ext^1_{\bmk Q}(M,\bmk Q)$ is a right $\bmk Q$-module, so that $\tau_{\bmk Q}(M) \coloneqq \Ext^1_{Q}(M,\bmk Q)^\ast$ is a left $\bmk Q$-module. There results an endofunctor, the {\it Auslander-Reiten translation},
\[ \tau_{\bmk Q} \colon \modd \bmk Q \lra \modd \bmk Q \] 
which is left exact, since $\bmk Q$ is hereditary. We denote the induced functor on $\rep(Q)$ by $\tau_{Q} \colon \rep(Q) \lra \rep(Q)$.
Similarly, we obtain the functor $\tau^{-1}_{Q} \colon\rep(Q) \lra \rep(Q)$ induced by 
$\tau^{-1}_{\bmk Q} \colon \modd \bmk Q \lra \modd \bmk Q ; M \mapsto  \Ext^1_{\bmk Q}(M^\ast,\bmk Q)$.

An indecomposable representation $M \in \rep(Q)$ is called {\it preprojective (preinjective)}, provided $\tau^n_{Q}(M) = 0$ (resp. $\tau^{-n}_{Q}(M) = 0$) for some $n \in \NN$. All other indecomposable representations are called {\it regular}. Since $Q$ is a wild quiver, and therefore not of Dynkin type, the three classes preprojective, preinjective and regular are mutually exclusive.\\
Given a representation $M \in \rep(Q)$, we let $\dimu M \coloneqq (\dim_\bmk M_i)_{i \in Q_0} \in \ZZ^n$ be its {\it dimension vector}. This gives rise to an isomorphism 
\[ \dimu \colon K_0(\rep(Q)) \lra \ZZ^n,\]
which identifies the Grothendieck group $K_0(\rep(Q))$ of $\rep(Q)$ with $\ZZ^n$.
Given $i \in Q_0$ we denote by $S(i)$ the simple representation corresponding to $i$ and by $P(i)$ and $I(i)$ its projective cover and injective hull, respectively. The sets $\{ \dimu P(i) \mid i \in \QQ_0\}$, $\{ \dimu I(i) \mid i \in \QQ_0\}$ are $\ZZ$-bases of $\ZZ^n$. The {\it Coxeter transformation} $\Phi_Q$ is the $\ZZ$-linear map $\Phi_Q \colon \ZZ^n \lra \ZZ^n$ with
\[ \Phi_Q(\dimu P(i)) = - \dimu I(i)\]
for all $i \in Q_0$. We have 
\[\dimu \tau_Q(M) = \Phi_Q(\dimu M) \ \text{and} \ \dimu \tau^{-1}_Q(N) = \Phi^{-1}_Q(\dimu M)\]
for $M,N$ indecomposable with $M \not \cong P(i), I \not \cong I(i)$ for all $i \in Q_0$. An arbitary non-zero representation $M \in \rep(Q)$ is called {\it preprojective, preinjective} or {\it regular}, provided all its indecomposable direct summands are preprojective, preinjective or regular, respectively. By definition the zero representation is preprojective, preinjective and regular.

\subsection{Wild Kronecker quivers}
We specialize our considerations to the family of wild Kronecker quivers. Throughout we let $r \in \NN_{\geq 3}$. The {\it (generalized) Kronecker quiver} with $r$ arrows, denoted by $K_r$, is the bipartite quiver with two vertices $1,2$ and arrows $\gamma_i \colon 1 \lra 2$ ($1 \leq i \leq r$). A representation $M$ over $K_r$ is a tuple $M = (M_1,M_2,(M(\gamma_i))_{1 \leq i \leq r})$ consisting of finite dimensional vector spaces $M_1,M_2$ and $\bmk$-linear maps $M(\gamma_i) \colon M_1 \lra M_2$ for each $1 \leq i \leq r$. A morphism $f \colon M \lra N$ of representations is a pair $(f_1,f_2)$ of $\bmk$-linear maps such that, for each $i \in \{1,\ldots,r\}$, the diagram
\[\xymatrix{
M_1 \ar^{M(\gamma_i)}[r] \ar^{f_1}[d]& M_2 \ar^{f_2}[d] \\
N_1 \ar^{N(\gamma_i)}[r] & N_2
}\]
commutes. The simple representations corresponding to the vertices $1$ and $2$ are denoted by $S_1 = S(1)$ and $S_2 = S(2)$, respectively. 
We let $A_r \coloneqq \bigoplus^r_{i=1} \bmk \gamma_i$ be the {\it arrow space} of $K_r$ and realize the path algebra of $K_r$ as 
\[ \bmk K_r = \begin{pmatrix}
\bmk & 0\\
A_r & \bmk
\end{pmatrix}.\]
We let 
\[ \langle - , - \rangle_r \colon \ZZ^2 \times \ZZ^2 \lra \ZZ, (x,y) \mapsto x_1 y_1 + x_2 y_2 - r x_1 y_2 \]
be the bilinear form given by $K_r$, with corresponding {\it Tits quadratic form} $q_r \colon \ZZ^2 \lra \ZZ, x \mapsto \langle x,x \rangle_r$. 

\subsection{Shift functors}
We denote by $\sigma_{K_r},\sigma^{-1}_{K_r} \colon \rep(K_r) \to \rep(K_r)$ the {\it shift functors}. These functors correspond to the BGP-reflection functors but take into account that the opposite quiver of $K_r$ is isomorphic to $K_r$, i.e. $D_{K_r} \circ \sigma_{K_r} \cong \sigma_{K_r}^{-1} \circ D_{K_r}$, where $D_{K_r} \colon \rep(K_r) \to \rep(K_r)$ denotes the standard duality.\\
For a representation $M \in \rep(K_r)$ we consider the $\bmk$-linear map
\[f_M \colon (M_1)^r \to M_2, (m_i) \mapsto \sum^{r}_{i=1} M(\gamma_i)(m_i).\]
Then $\sigma_{K_r}(M)$ is by definition the representation \[ (\sigma_{K_r}(M)_1,\sigma_{K_r}(M)_2,(\sigma_{K_r}(M)(\gamma_i))_{1 \leq i \leq r}) = (\ker f_M,M_1,(\pi_{i}|_{\ker f_M})_{1 \leq i \leq r}),\]
where $\pi_{i} \colon (M_1)^r \to M_1$ is the projection onto the $i$-th component for each $i \in \{1,\ldots,r\}$. Recall that $\sigma_{K_r}$ induces an equivalence
\[ \sigma_{K_r} \colon \rep_2(K_r) \lra \rep_1(K_r)\]
between the full subcategories $\rep_i(K_r)$ of $\rep(K_r)$, whose objects do not have have any direct summands isomorphic to $S_i$. By the same token, $\sigma^{-1}_{K_r}$ is a quasi-inverse of $\sigma_{K_r}$. The map
\[ \sigma_r \colon \ZZ^2 \lra \ZZ^2 ; (x,y) \mapsto (rx - y,x)\]
is invertible and satisfies
\[ \dimu \sigma_{K_r}(M) = \sigma_{r}(\dimu M) \ \text{and} \ \dimu \sigma^{-1}_{K_r}(N) = \sigma^{-1}_{r}(\dimu N)\]
for all $M \in \rep_2(K_r)$ and $N \in \rep_1(K_r)$. Moreover, we have $\sigma_{K_r} \circ \sigma_{K_r} \cong \tau_{K_r}$ and $\sigma^2_r = \Phi_r \coloneqq \Phi_{K_r}$.

\subsection{Indecomposable representations and Kac's Theorem} 
The preprojective and preinjective indecomposable Kronecker representations are well-understood: We define $P_0 \coloneqq S_2$ and $P_i \coloneqq \sigma^{-1}_{K_r}(P_{i-1})$ for all $i \geq 1$. The representations $P_i$ form a complete list of representatives of the isomorphism classes of indecomposable preprojective Kronecker representations. By the same token, a complete list of representatives of the isomorphism classes of indecomposable preinjective Kronecker representations is given by $I_i \coloneqq D_{K_r}(P_i)$, $i \in \NN_0$.  
Since $\sigma_r$ and $\sigma_r^{-1}$ leave the Tits form invariant and $q_r(1,0) = 1 = q_r(0,1)$, this shows that $q_{K_r}(\dimu N) = 1$ for $N$ indecomposable and preprojective or preinjective.
We let $L_r \coloneqq \frac{r + \sqrt{r^2-4}}{2}$ and note that $L_r$ and $\frac{1}{L_r}$ are the roots of the polynomial $f_r \coloneqq X^2 - rX + 1\in \ZZ[X]$. Therefore they satisfy the equation $\frac{1}{L_r} = r - L_r$. Moreover, we have $r- 1 < L_r < r$ since $r \geq 3$. We let $\delta \colon \ZZ^2 \lra \ZZ^2 ; (a,b) \mapsto (b,a)$ be twist function on $\ZZ^2$.
Then we have
\[ \dimu P_i = (a_i,a_{i+1}) = \delta(\dimu I_i), \ \text{where for all} \ i \in \NN_0 \ a_i \coloneqq \frac{(L_r)^i - (\frac{1}{L_r})^i}{\sqrt{r^2-4}}.\]
We recall a simplified version of Kac's Theorem (see \cite[Thm.B]{Kac82} and \cite[Thm.3]{Rin76}) and an immediate consequence thereof that suffice for our purposes.

\begin{Theorem}(Kac's Theorem for $K_r$) \label{1.1}
Let $\delta \in \NN^2_0 \setminus \{0\}$.
\begin{enumerate}
    \item If $\delta = \dimu M$ for some indecomposable $M \in \rep(K_r)$, then $q_r(\delta) \leq 1$.
    \item If $q_{r}(\delta) = 1$, then there is a, up to isomorphism, unique indecomposable representation $M \in \rep(K_r)$ such that $\dimu M = \delta$. The representation $M$ is preprojective or preinjective and preprojective if and only if $\delta_1 \leq \delta_2$.
    \item If $q_{r}(\delta) \leq 0$, then there are infinitely many pairwise non-isomorphic indecomposable representations with dimension vector $\delta$, each being regular.
\end{enumerate}
\end{Theorem}

\begin{corollary}\label{1.2}
Let $M \in \rep(K_r)$ be indecomposable. The following statements hold.
\begin{enumerate}
    \item $M$ is preprojective if and only if $\dim_\bmk M_1 < \frac{1}{L_r} \dim_\bmk M_2$.
    \item $M$ is regular if and only if $\frac{1}{L_r} \dim_\bmk M_2 < \dim_\bmk M_1 < L_r \dim_\bmk M_2$.
    \item $M$ is preinjective if and only if $L_r \dim_\bmk M_2 <  \dim_\bmk M_1$.
\end{enumerate}
\end{corollary}

\section{Geometric considerations and restrictions on dimension vectors}

Throughout this section $d$ denotes a natural number with $1 \leq d < r$. For $(x,y) \in \NN^2_0$ we write
\[ \Delta_{(x,y)}(d) \coloneqq y - d x \ \text{and} \ \nabla_{(x,y)}(d) \coloneqq d y- x.\]
For a representation $M \in \rep(K_r)$, or vector spaces $M_1,M_2 \in \modd \bmk$, we define 
\[\Delta_{M}(d) \coloneqq \Delta_{(M_1,M_2)}(d) \coloneqq \Delta_{(\dim_\bmk M_1,\dim_\bmk M_2)}(d) \ \text{and}\]
\[\nabla_{M}(d) \coloneqq \nabla_{(M_1,M_2)}(d) \coloneqq \nabla_{(\dim_\bmk M_1,\dim_\bmk M_2)}(d).\]

\subsection{Relative projective representations and vector bundles}

Let $M \in \rep(K_r)$ be a representation with {\it structure map}
\[ \psi_M \colon A_r \otimes_{\bmk} M_1 \to M_2 \ ; \ \sum^r_{i=1} \gamma_i \otimes m \mapsto \sum^r_{i=1} M(\gamma_i)(m).\] 
We say that $M \in \rep(K_r)$ is {\it relative $d$-projective}, provided that $\psi_M|_{\fv \otimes M_1}$ is injective for each $\fv \in \Gr_d(A_r)$, where $\Gr_d(A_r)$ denotes the Grassmannian of $d$-dimensional subspaces of $A_r$. 

\bigskip

\begin{Remark}\label{2.1} The terminology "relative $d$-projective" is motivated  by the fact that $\psi_M|_{\fv \otimes M_1}$ is injective if and only the restriction of the $\bmk K_r$-module $M$ to the subalgebra 
\[ \bmk K_d \cong \bmk \fv \coloneqq \begin{pmatrix}
\bmk & 0 \\
\fv  & \bmk
\end{pmatrix} \subseteq \bmk K_r \]
is projective (cf. \cite[2.1.5]{BF24}).
\end{Remark}

\bigskip

The full subcategory of relative $d$-projective representations is denoted by $\rep_{\proj}(K_r,d)$. This category is a torsion-free class (see \cite[2.2.1]{BF24}) closed under $\sigma_{K_r}^{-1}$ and gives rise to special vector bundles (locally free coherent sheaves) on $\Gr_d(A_r)$. In the following we recall results and definitions from \cite{BF24} and \cite{AM15}.\\
Let $\Vect(\Gr_d(A_r))$ be the category of vector bundles on $\Gr_d(A_r)$ with structure sheaf $\cO_{\Gr_d(A_r)}$. Moreover, let $\cU_{(r,d)}$ be the {\it universal vector bundle} of $\Gr_{d}(A_r)$. A locally free sheaf $\cF \in \Coh(\Gr_d(A_r))$ is called {\it Steiner bundle}, provided there exist vector spaces $V_1,V_2$ and a short exact sequence
\[ 0 \lra V_1 \otimes_\bmk \cU_{(d,r)} \lra V_2 \otimes_\bmk \cO_{\Gr_d(A_r)} \lra \cF \lra 0.\]
\indent We denote by $\StVect(\Gr(A_r))$ the full subcategory of Steiner bundles on $\Gr_d(A_r)$. The following result is proven in \cite[2.3.2, 3.3.2, 3.3.3]{BF24}. The proof of (2) elaborates on \cite[2.4]{AM15}, where the result was first shown for algebraically closed fields of characteristic zero.

\begin{Theorem}\label{2.2}
The following statements hold. 
\begin{enumerate}
    \item There exists a fully faithful and exact functor
\[ \widetilde{\Theta}_d \colon \rep_{\proj}(K_r,d) \lra \Vect(\Gr_d(A_r))\]
with essential image $\StVect(\Gr_d(A_r))$. Moreover, there is a short exact sequence
\[ 0 \lra M_1 \otimes_\bmk \cU_{(r,d)} \lra M_2 \otimes_\bmk \cO_{\Gr_d(A_r)} \lra \widetilde{\Theta}_d(M) \lra 0\]
for each $M \in \rep_{\proj}(K_r,d)$.
\item For each Steiner bundle
\[ 0 \lra V_1 \otimes_\bmk \cU_{(r,d)} \lra V_2 \otimes_\bmk \cO_{\Gr_d(A_r)}\lra \cF \lra 0 \] we have $\rk(\cF) \geq  \min \{d(r-d),(\dim_\bmk V_1)(r-d)\}$.
\item Let $M \in \rep_{\proj}(K_r,d)$, then $\Delta_M(d) \geq \min \{d(r-d),\dim_\bmk M_1(r-d)\}$.
\end{enumerate}
\end{Theorem}

We record direct consequences of \cref{2.2} that will be needed later on when we study elementary representations.

\begin{corollary}\label{2.3} Let $M \in  \rep_{\inj}(K_r,d) \coloneqq D_{K_r}(\rep_{\proj}(K_r,d))$, then  
\[-\nabla_M(d) \geq \min \{d(r-d),\dim_\bmk M_2(r-d)\}.\]
\end{corollary}
\begin{proof}
Since $M \in \rep_{\inj}(K_r,d)$, we have $D_{K_r}(M) \in \rep_{\proj}(K_r,d)$ and therefore
\begin{align*}
    -\nabla_M(d) &= \dim_\bmk M_1 - d \dim_\bmk M_2 =\Delta_{D_{K_r}(M)}(d) \\
    &\geq \min \{d(r-d),\dim_\bmk (D_{K_r}(M))_1(r-d)\} = \min \{d(r-d),\dim_\bmk M_2(r-d)\}.
\end{align*} 
\end{proof}

\begin{corollary}\label{2.4} 
The following statements hold.
\begin{enumerate}
    \item Let $M \in \rep_{\proj}(K_r,d)$ with $\dim_\bmk M_1 \leq d$, then $M$ is projective.
    \item Let $M \in \rep_{\inj}(K_r,d)$ with $\dim_\bmk M_2 \leq d$, then $M$ is injective.
\end{enumerate}
\end{corollary}
\begin{proof}
\begin{enumerate}
    \item Let $N$ be an indecomposable direct summand of $M$ not isomorphic to $P_0$. Then $N_1 \neq 0$, $N \in \rep_{\proj}(K_r,d)$ and $\dim_\bmk N_1 \leq d$.
We have 
\[ \dim_\bmk N_2 - d \dim_\bmk N_1 =  \Delta_N(d) \geq \min\{d(r-d),\dim_\bmk N_1(r-d)\} = \dim_\bmk N_1(r-d)\]
and conclude $\dim_\bmk N_2 \geq r \dim_\bmk N_1$. Since $\dim_\bmk N_1 \neq 0$, we also have a projective resolution
\[ 0 \lra P_0^{r \dim_\bmk N_1- \dim_\bmk N_2} \lra P_1^{\dim_\bmk N_1} \lra N \lra 0\]
and conclude $r \dim_\bmk N_1- \dim_\bmk N_2 = 0$ as well as $N \cong P_1^{\dim_\bmk N_1}$.
\item This follows from duality.
\end{enumerate}
\end{proof}

\subsection{The equal socle property and connections to relative projective representations}\label{Section2.2}

Constant rank type modules have been defined and studied in \cite{CFP12} in the context of elementary abelian $p$-groups over fields of characteristic $p$ as a generalization of constant Jordan type modules. Inspired by these considerations, representations with the equal socle type have been introduced in \cite{Bis20} for Kronecker representations over fields of arbitrary characteristic. It is the aim of this section characterize dimension vectors that admit representations with the equal socle property. This description plays a crucial role in \cref{Section:3}, when we determine the dimension vectors admitting an elementary representation.\\
We obtain these restrictions with the help of a recent result of Reineke in the framework of generic representations for Kronecker representations.
%As one application we give a new proof of \cref{2.2}(2),(3).
Let $M \in \rep(K_r)$ and $\fv \in \Gr_d(A_r)$. Given $a = \sum^r_{i=1} \alpha_i \gamma_i \in A_r$ we denote by $a_M \colon M_1 \lra M_2$ the $\bmk$-linear map
\[ a_M \colon M_1 \lra M_2 \ ; \ m \mapsto a.m \coloneqq \sum^r_{i=1} \alpha_i M(\gamma_i)(m).\]

\begin{Definition}(cf. \cite[2.3]{Bis20}). 
A representation $M \in \rep(K_r)$ has the {\it equal $d$-socle property}, provided $\{0\} = \bigcap_{a \in \fv} \ker a_{M}$
for all $\fv \in \Gr_d(A_r)$.
\end{Definition}

We note that $\rep_{\esp}(K_r,d)$ and $\rep_{\proj}(K_r,d)$ are closed under subrepresentations and direct sums. Relative projective representations and representations with the equal socle property are closely related: 

\begin{Lemma} \label{2.5} 
Let $N \in \rep(K_r)$ and $1 \leq d < r$. The following statements are equivalent.
\begin{enumerate}
    \item $N \in \rep_{\proj}(K_r,d)$.
    \item $\sigma_{K_r}(N) \in \rep_{\esp}(K_r,r-d)$.
\end{enumerate}
\end{Lemma}
\begin{proof}
We define $M \coloneqq \sigma_{K_r}(N)$. Clearly, $P_0 \in \rep_{\proj}(K_r,d)$ and $\sigma_{K_r}(P_0) = \{0\} \in \rep_{\esp}(K_r,r-d)$. Since the involved categories are closed under direct sums and summands, we may assume that $N$ does not have $P_0$ as a direct summand.

(1) $\implies$ (2). We assume that $N \not\in \rep_{\proj}(K_r,d)$. By definition we find $\fv \in \Gr_{d}(A_r)$ such that
\[ \psi_N|_{\fv \otimes_{\bmk} N_1} \colon \fv \otimes_{\bmk}  N_1 \lra N_2 ; a \otimes n \mapsto a.n\]
is not injective. Let $(a_1,\ldots,a_{d})$ be a basis of $\fv$ and $x = \sum^{d}_{j=1} a_j \otimes n_j$ be a non-zero element in $\ker \psi_N|_{\fv \otimes N_1}$. We write $a_j = \sum^r_{i=1} \beta_{ij} \gamma_i$ for $1 \leq j \leq d$ and set $n_i' \coloneqq \sum^d_{j=1} \beta_{ij} n_j$ for $1 \leq i \leq r$. By definition we have 
\[ x = \sum^{d}_{j=1} \sum^r_{i=1} \beta_{ij} \gamma_i \otimes n_j = \sum^r_{i=1} \gamma_i \otimes  \sum^{d}_{j=1}  {\beta_{ij} n_j}= \sum^r_{i=1} \gamma_i \otimes n_i'.\]
Recall that \[M_2 = N_1, M_1 = \ker (N_1^r \to N_2 ; (y_i)_{1 \leq i \leq r} \mapsto \sum^r_{i=1} \gamma_i.y_i)\] and $\gamma_j.((y_i)_{1 \leq i \leq r}) = M(\gamma_j)((y_i)_{1 \leq i \leq r}) = y_j$ for $j \in \{1,\ldots,r\}$. We have
\[ 0 =   \psi_N|_{\fv \otimes N_1}(x) =  \sum^r_{i=1} \gamma_i.n_i',\]
which shows $m \coloneqq (n_i')  \in M_1 \setminus \{0\}$. Let $A \coloneqq \{ \delta \in \bmk^r \mid \sum^r_{i=1} \delta_i n_i' = 0\}$. Since $\sum^r_{i=1} \bmk n_i' \subseteq \sum^d_{i=1} \bmk n_i$, we have $\dim_\bmk A \geq r - d$. We fix a subspace $B \subseteq A$ of dimension $r-d$. Let $\fu \coloneqq \{ \sum^r_{i=1} \delta_i \gamma_i \mid \delta \in B\} \in \Gr_{r-d}(A_r)$. Let $a = \sum^r_{i=1} \delta_i \gamma_i \in \fu$, then $a.m = \sum^r_{i=1} \delta_i \gamma_i.m = \sum^r_{i=1} \delta_i n_i' = 0$. Hence $0 \neq m \in \bigcap_{a \in \fu} \ker a_M$ and $M \not\in \rep_{\esp}(K_r,r-d)$.

(2) $\implies$ (1). Assume that $M \not\in \rep_{\esp}(K_r,r-d)$. We find $\fu \in \Gr_{r-d}(A_r)$ and $0 \neq m \in \bigcap_{a \in \fu} \ker a_M \setminus \{0\}$. By definition we have $m = (n_1,\ldots,n_r) \in N_1^r \setminus \{0\}$ and $0 = \sum^r_{i=1} \gamma_i.n_i$. Let $a = \sum^r_{i=1} \lambda_i \gamma_i \in \fu$, then $0 = a_M(m) = \sum^r_{i=1} \lambda_{i} M(\gamma_i)(m) = \sum^r_{i=1} \lambda_{i} n_i$.
Hence $\{ \delta \in \bmk^r \mid \sum^r_{i=1} \delta_i n_i = 0\}$ is a vector space of dimension at least $r-d$ and $\sum^r_{i=1}\bmk n_i$ a vector space of dimension at most $d$. Let $(x_1,\ldots,x_m)$ be a basis of $\sum^r_{i=1} \bmk n_i$. We write $n_i = \sum^m_{j=1}  \lambda_{ij} x_j$ for $1 \leq i \leq r$ and let $b_l \coloneqq \sum^r_{j=1} \lambda_{jl} \gamma_j$ for $1 \leq l \leq m$. Let $\fv \in \Gr_d(A_r)$ such that $\sum^m_{l=1} \bmk b_l \subseteq \fv$. We have $0 \neq \sum^m_{l=1} b_l \otimes x_l \in \fv \otimes_{\bmk} N_1$ and
\begin{align*}
    \psi_{N}|_{\fv \otimes N_1}(\sum^m_{l=1} b_l \otimes x_l) &= \sum^m_{l=1} (\sum^r_{j=1}  \lambda_{jl} \gamma_j).x_l =  \sum^r_{j=1} \gamma_j.\sum^m_{l=1}  \lambda_{jl} x_l  = \sum^r_{j=1} \gamma_j.n_j = 0.
\end{align*}
Hence $N \not\in \rep_{\proj}(K_r,d)$.
\end{proof}

\subsection{Generic representations and applications}

Let $(V_1,V_2)$ be a pair vector spaces. We denote by $\rep(K_r;V_1,V_2) \coloneqq \Hom_\bmk(V_1,V_2)^r$ the affine variety of representations of $K_r$ on $(V_1,V_2)$. Given $\cS \subseteq \rep(K_r)$ and $\cT \subseteq \rep(K_r;V_1,V_2)$ we define \[ \cS \cap \cT \coloneqq \cT \cap \cS \coloneqq \{ g \in \cT  \mid (V_1,V_2,g) \in \cS\} \subseteq \rep(K_r;V_1,V_2).\]
Let $\bd \coloneqq (\dim_\bmk V_1,\dim_\bmk V_2)$. For $\be \leq \bd \in \NN^2_0$ (componentwise) we let $\rep(K_r;V_1,V_2)_{\be}$ be the Zariski-closed subset (cf. \cite[3.1]{Sch92}) of $\rep(K_r;V_1,V_2)$ consisting of all representations admitting a subrepresentation of dimension vector $\be$. We write $\be \hookrightarrow \bd$ if $\rep(K_r;V_1,V_2)_{\be} = \rep(K_r;V_1,V_2)$. Otherwise we write $\be \not \hookrightarrow \bd$. Schofield gave in \cite{Sch92} a criterion (in a way more general setting) in characteristic zero to decide whether $\be \hookrightarrow \bd$ holds. Crawley-Boevey extended this criterion in \cite{CB96} to positive characteristic. The statemenent in the Kronecker setting reads as follows:

\begin{Theorem}[Crawley-Boevey, Schofield]\label{2.6}
We have $\be \hlra \bd$ if and only if $\langle \bff,\bd-\be\rangle_r \geq 0$ for all $\bff \hlra \be$.
\end{Theorem}

For imaginary roots the statement can be simplified:
\begin{proposition}(see \cite[3.4]{Rei22})\label{2.7}
Assume that $q_r(\bd) \leq 0$. The following statements are equivalent for $\be \in \NN^2_0$ with $\be \leq \bd$.
\begin{enumerate}
    \item $\be \hookrightarrow \bd$.
    \item $\langle \be,\bd-\be\rangle_r \geq 0$.
\end{enumerate}
\end{proposition}

We adapt the proof of Reineke to show:

\begin{proposition}\label{2.8}
The following statements are equivalent for $\be \in \NN^2_0$ with $\be \leq \bd$ and $q_r(\be) \leq 1$.
\begin{enumerate}
    \item $\be \hlra \bd$.
    \item $\langle \be,\bd - \be \rangle_r \geq 0$.
\end{enumerate}
\end{proposition}
\begin{proof}
(1) $\implies$ (2). Apply \cref{2.6} for $\bff = \be$.

(2) $\implies$ (1). Let $\bff \hlra \be$. In view of \cref{2.6} it suffices to show that $\langle \bff,\bd-\be \rangle_r \geq 0$. 
Since $q_r(\be) \leq 1$ holds, $\be$ is a Schur root (see for example \cite[1.2.2]{BF24}). Hence \cite[6.1]{Sch92} implies $0 <  \langle \bff,\be\rangle_r - \langle \be,\bff\rangle_r = r(e_1 f_2 - e_2 f_1)$. In particular, $e_1 \neq 0$ and $f_2 > \frac{e_2 f_1}{e_1}$. We conclude with $d_2 - e_2 \geq 0$
\begin{align*}
    \langle \bff,\bd-\be\rangle_r &= f_1(d_1-e_1-r (d_2 - e_2)) + f_2(d_2-e_2) \geq f_1(d_1-e_1-r (d_2 - f_2)) + \frac{f_1 e_2}{e_1}(d_2-e_2)\\
    &= \frac{f_1}{e_1} \langle \be,\bd-\be\rangle_r \geq 0.
\end{align*} 
\end{proof}

In order to use \cref{2.8}, we give a characterization of $\rep_{\esp}(K_r,d)$ and $\rep_{\proj}(K_r,d)$ in terms of abscence of subrepresentations:
%In the following we show that $\rep_{\esp}(K_r,d) \cap \rep(K_r;V_1,V_2) \subseteq \rep(K_r;V_1,V_2)$ is open and that $\rep_{\esp}(K_r,d) \cap \rep(K_r;V_1,V_2) \neq \emptyset$ implies $V_1 = 0 $ or $\nabla_{(V_1,V_2)}(d) \geq \Delta(d) = d(r-d)$.

\begin{proposition}\label{2.9}
Let $M \in \rep(K_r)$. 
\begin{enumerate}
    \item The following statements are equivalent.
    \begin{enumerate}
        \item[(i)] $M \not\in \rep_{\esp}(K_r,d)$.
        \item[(ii)] There exists $a \in \{0,\ldots,r-d\}$ and a subrepresentation $X \subseteq M$ with dimension vector $(1,a)$.
    \end{enumerate}
    \item  The following statements are equivalent.
   \begin{enumerate}
     \item[(i)] $M \not\in \rep_{\proj}(K_r,d)$.
     \item[(ii)] There exist $a \in \{1,\ldots,d\}$, $a' \in \{0,\ldots,ar-1\}$ and a subrepresentation $X \subseteq M$ with dimension vector $(a,a')$.
 \end{enumerate}
\end{enumerate}
\end{proposition}

\begin{proof}
\begin{enumerate}
    \item (i) $\implies$ (ii). Let $\fv \in \Gr_d(A_r)$ and $0 \neq x \in \bigcap_{a \in \fv} \ker a_M$. We denote by $X$ the representation generated by $x$. Let $\fu \in \Gr_{r-d}(A_r)$ such that $\fu \oplus \fv = A_r$. Then $X_2 = \im \psi_{M}(\fv \otimes_{\bmk} X_1) + \im \psi_{M}(\fu \otimes_{\bmk} X_1) = \im \psi_{M}(\fu \otimes_{\bmk} X_1)$. Since $\dim_{\bmk} X_1 = 1$, we obtain $ \psi_{M}(\fu \otimes_{\bmk} X_1) \leq \dim_{\bmk} \fu = r-d$.

(ii) $\implies$ (ii). Let $x \in X_1 \setminus \{0\}$. Then $x$ generates an indecomposable representation $\langle x \rangle \subseteq X$ with $\dimu \langle x \rangle = (1,u)$ for some $0 \leq u \leq r - d$ and
\[ \psi_{M}|_{A_r \otimes_{\bmk} \bmk x} \colon  A_r \otimes_{\bmk} \bmk x  \lra (\langle x \rangle)_2 ; a \otimes m \mapsto a.m \]
is surjective. We have $\dim_{\bmk} \ker \psi_{M}|_{A_r \otimes_{\bmk} \bmk x} = r - u\geq  r - (r-d) = d$. Hence we find $\fv \in \Gr_d(A_r)$ such that $\fv.x = \{0\}$, $0 \neq x \in \bigcap_{a \in \fv} a_M$ and $M \not\in \rep_{\esp}(K_r,d)$.
    \item  (i) $\implies$ (ii). By definition we find $\fv \in \Gr_d(A_r)$ such that $\psi_M|_{\fv\otimes_{\bmk} M_1} \colon \fv \otimes_{\bmk} M_1 \lra M_2$ is not injective. We fix a basis $(v_1,\ldots,v_d)$ of $\fv$ and an element $0 \neq x = \sum^d_{i=1} v_i \otimes m_i$ in the kernel of $\psi_M|_{\fv\otimes_{\bmk} M_1}$. We consider the module $X \subseteq M$ generated by $\{m_1,\ldots,m_d\}$, then $1 \leq \dim_{\bmk} X_1 \leq d$. Let $\fu \in \Gr_{r-d}(A_r)$ be a direct complement of $\fv$ in $A_r$. We have
\begin{align*}
    \dim_{\bmk} X_2 &\leq \dim_{\bmk} \psi_{M}(\fv \otimes_{\bmk} X_1) +\dim_{\bmk} \psi_{M}(\fu \otimes_{\bmk} X_1)\\
    &\leq d\dim_{\bmk}X_1 - 1 + (r-d)\dim_{\bmk} X_1\\
    &= r \dim_{\bmk} X_1 -1 .
\end{align*} 
(ii) $\implies$ (i). We write $X = Y \oplus P_0^{\ell}$ such that $P_0$ is not a direct summand of $Y$. Then 
\[ \psi_M|_{A_r \otimes_{\bmk} Y_1} \colon A_r \otimes_{\bmk} Y_1 \lra Y_2\]
is surjective. We have $\dim_{\bmk} Y_1 = a$, $\dim_{\bmk} Y_2 \leq ar-1$ and obtain $\dim_{\bmk} \ker(\psi_M|_{A_r \otimes_{\bmk} Y_1}) \geq ra-(ar-1) = 1$. Let $(v_1,\ldots,v_a)$ be a basis of $Y_1$. We find $0 \neq x = \sum^a_{i=1} y_i \otimes v_i\in \ker(\psi_M|_{A_r \otimes_{\bmk} Y_1})$ and $\fv \in \Gr_d(A_r)$ containing $y_1,\ldots,y_a$. Therefore $0 \neq x \in \ker \psi|_{\fv \otimes_{\bmk} M_1}$.
\end{enumerate}
\end{proof}

\begin{Remark}\label{2.10}
Note that the subrepresentations $X$ in (1) and (2) are not in $\rep_{\esp}(K_r,d)$ and $\rep_{\proj}(K_r,r-d)$, respectively. In particular, they are not preprojective.
\end{Remark}

\begin{Theorem}\label{2.11} Let $V_1,V_2$ be vector spaces such that $V_1 \oplus V_2 \neq 0$. The following statements hold.
\begin{enumerate}
    \item The set $\rep_{\esp}(K_r,d) \cap \rep(K_r;V_1,V_2)$ is open in $\rep(K_r;V_1,V_2)$.
   % \item For $i \in \NN$ we have $\nabla_{P_i}(d) \geq d(r-d)$.
    \item The following statements are equivalent.
    \begin{enumerate}
        \item[(i)] $\rep_{\esp}(K_r,d) \cap \rep(K_r;V_1,V_2) \neq \emptyset$.
        \item[(ii)] $V_1 = 0$ or $\nabla_{(V_1,V_2)}(d) \geq d(r-d)$.
    \end{enumerate}
\end{enumerate}
\end{Theorem}
\begin{proof}
\begin{enumerate}
    \item By \cref{2.9} we have \[\rep_{\esp}(K_r,d) \cap \rep(K_r;V_1,V_2)  = \rep(K_r;V_1,V_2) \setminus \bigcup_{i=0}^{r-d} \rep(K_r;V_1,V_2)_{(1,i)}.\]
   \item 
    (i) $\implies$ (ii) Assume that $\rep_{\esp}(K_r,d) \cap \rep(K_r;V_1,V_2) \neq \emptyset$. We assume that $V_1 \neq 0$. Then $\dim_{\bmk} V_2 > r -d$ by \cref{2.9}. Another application of \cref{2.9} implies $(1,r-d) \not \hookrightarrow (\dim_{\bmk} V_1,\dim_{\bmk} V_2)$. We have $q_r(1,r-d)\leq 1$ and conclude with \cref{2.8}
   \begin{align*}
        0 &> \langle (1,r-d),(\dim_{\bmk} V_1,\dim_{\bmk} V_2)-(1,r-d)\rangle_r\\
        &= \dim_{\bmk}V_1 - d \dim_{\bmk} V_2 - (1-d(r-d)) \\
        &= -\nabla_{(V_1,V_2)}(d) + d(r-d) - 1.
   \end{align*}
    (ii) $\implies$ (i). If $V_1 = 0$, the statement is clear. Hence we assume $\dim_\bmk V_1 \neq 0$ and $\nabla_{(V_1,V_2)}(d) \geq d(r-d)$. We have $d\dim_\bmk V_2 \geq \nabla_{(V_1,V_2)}(d) \geq d(r-d)$ and conclude $\dim_{\bmk} V_2 \geq r - d$. Hence $(1,r-d) \leq \dimu(V_1,V_2)$ with
   \[ \langle (1,r-d),\dimu(V_1,V_2)-(1,r-d)\rangle_r = -\nabla_{(V_1,V_2)}(d) + d(r-d) - 1 \leq  - 1.\]
   Since $q_r(1,r-d) \leq 1$, we conclude with \cref{2.8} that $\rep(K_r;V_1,V_2) \setminus \rep(K_r;V_1,V_2)_{(1,r-d)}$ is non-empty. Note that $\dim_{\bmk} V_2 \geq r-d$ and $V_1 \neq 0$ imply $\rep(K_r;V_1,V_2)_{(1,r-d)} = \bigcup_{i=0}^{r-d} \rep(K_r;V_1,V_2)_{(1,i)}$. Now we apply \cref{2.9}.
\end{enumerate}
\end{proof}

Recall that a representation $M \in \rep(K_r)$ is called {\it brick} if $\End_{K_r}(M) \cong \bmk$. Clearly, bricks are indecomposable. Given $M \in \rep(K_r)$ indecomposable, Kac's Theorem implies $q_{r}(\dimu M) \leq 1$. Therefore the following result describes all dimension vectors that can be realized by indecomposable elements in $\rep_{\esp}(K_r,d)$.

\begin{corollary}\label{2.12}
Let $V_1,V_2$ be a pair of vector spaces such that $V_1\oplus V_2 \neq 0$ and $q_r(\dimu(V_1,V_2)) \leq 1$. The following statements are equivalent.
      \begin{enumerate}
                \item[(i)] $\rep_{\esp}(K_r,d) \cap \rep(K_r;V_1,V_2) \neq \emptyset$.
                \item[(ii)] $\rep_{\esp}(K_r,d) \cap \rep(K_r;V_1,V_2)$ is a dense open subset $\rep(K_r;V_1,V_2)$.
                \item[(iii)] There exists a brick $N \in \rep_{\esp}(K_r,d)$ with dimension vector $\dimu N = \dimu (V_1,V_2)$.
                \item[(iv)] $\nabla_{(V_1,V_2)}(d) \geq d(r-d)$ or $\dimu (V_1,V_2) = (0,1)$.
       \end{enumerate}
  \end{corollary}
\begin{proof}
(i) $\implies$ (ii). This is clear since  $\rep_{\esp}(K_r,d) \cap \rep(K_r;V_1,V_2)$ is open in $\rep(K_r;V_1,V_2)$ by \cref{2.11} and $ \rep(K_r;V_1,V_2)$ is irreducible.\\
(ii) $\implies$ (iii). Since $q_r(\dimu (V_1,V_2)) \leq 1$ and $V_1 \oplus V_2 \neq 0$, we know from \cite[1.2.2]{BF24} that the open set
\[ \cB{(V_1,V_2)} \coloneqq \{ g \in \rep(K_r;V_1,V_2) \mid (V_1,V_2,g) \ \text{is a brick}\}\]
is dense in $\rep(K_r;V_1,V_2)$. Hence $\cB{(V_1,V_2)} \cap \rep_{\esp}(K_r,d)$ lies also dense $\rep(K_r;V_1,V_2)$ and is in particular non-empty. \\
(iii) $\implies$ (iv). Follows from \cref{2.11} and the fact that a representation with dimension vector $(0,\dim_{\bmk} V_2)$ is isomorphic to $P_0^{\dim_{\bmk} V_2}$.\\
(iv) $\implies$ (i) Follows from \cref{2.11}.
\end{proof}

Now we have the tools to give an alternative proof of \cref{2.2}(2).

\begin{corollary}(cf. \cite[2.3.2, 3.3.2]{BF24}, \cite[2.4]{AM15}) \label{2.13}
 Let $V_1,V_2$ be vector spaces such that $V_1 \oplus V_2 \neq 0$. The following statements hold.
\begin{enumerate}
\item The set $\rep_{\proj}(K_r,d) \cap \rep(K_r;V_1,V_2)$ is open in $\rep_{\proj}(K_r,d)$. 
\item The following statements are equivalent.
\begin{enumerate}
    \item[(i)] $\rep_{\proj}(K_r,d) \cap \rep(K_r;V_1,V_2) \neq \emptyset$.
    \item[(ii)] $\Delta_{(V_1,V_2)}(d) \geq \min\{d(r-d),\dim_{\bmk} V_1 (r-d)\}$.
\end{enumerate}
\item Let $\cF$ be a Steiner bundle on $\Gr_d(A_r)$ with resolution
\[ 0 \lra V_1 \otimes_\bmk \cU_{(r,d)} \lra V_2 \otimes_\bmk \cO_{\Gr_d(A_r)} \lra \cF \lra 0,\]
then $\rk(\cF) \geq \min \{ d(r-d),\dim_\bmk V_1 (r-d)\}$.
\end{enumerate}
\end{corollary}
\begin{proof}
\begin{enumerate}
\item In view of \cref{2.9} the set 
    \[\rep_{\proj}(K_r,d) \cap \rep(K_r;V_1,V_2) = \rep(K_r;V_1,V_2) \setminus \bigcup_{\be \in \fM}  \rep(K_r;V_1,V_2)_{\be}\]
    for $\fM := \{ (a,a') \mid a \in \{0,\ldots,d\}, a' \in \{0,\ldots,ad-1\}\}$ is open. 
    \item 
        (i) $\implies$ (ii). Let $M \in \rep_{\proj}(K_r,d)$. We write $M = P_0^a \oplus P_1^b \oplus N$ such that $P_0,P_1 \nmid N$. If $N \neq 0$, we have $\dimu \sigma_{K_r}(M) = \sigma_{K_r}(N) \oplus P_0^b$ and \cref{2.5} implies $\sigma_{K_r}(N) \in \rep_{\esp}(K_r,r-d)$. We have
\begin{align*}
    \dimu \sigma_{K_r}(N) &= \sigma_r(\dim_\bmk M_1 - b,\dim_\bmk M_2 - a - rb) \\
    &= (r(\dim_\bmk M_1-b) - \dim_\bmk M_2 + a + rb,\dim_\bmk M_1-b)\\
    &= (r \dim_\bmk M_1 - \dim_\bmk M_2 + a,\dim_\bmk M_1-b).
\end{align*}
Since $N$ is not projective, we have $\sigma_{K_r}(N)_1\neq 0$ and conclude with \cref{2.11}
\begin{align*}
    d(r-d)  &\leq \nabla_{\sigma_{K_r}(N)}(r-d) \\
    &=(r-d)(\dim_\bmk M_1 - b) - (r \dim_\bmk M_1 - \dim_\bmk M_2 + a)\\
    &= \dim_\bmk M_2  - d \dim_\bmk M_1 - b(r-d)  -a = \Delta_M(d) - b(r-d) -a.
\end{align*} 
Hence 
\[ d(r-d) \leq d(r-d) + b(r-d) +a \leq \Delta_M(d).\]
Now assume that $N = 0$, i.e. $M$ is projective. Then $\Delta_M(d) = b(r-d)+a \geq b(r-d) = \dim_\bmk M_1 (r-d)$.

(ii) $\implies$ (i). At first we consider the case $\Delta_{(V_1,V_2)}(d) \geq \dim_{\bmk} V_1(r-d)$. Then we have $\dim_{\bmk} V_2 \geq r \dim_{\bmk} V_1$, i.e. $\Delta_{(V_1,V_2)}(r) \geq 0$. Since $(\dim_{\bmk} V_1) P_1 \oplus \Delta_{(V_1,V_2)}(r) P_0 \in \rep_{\proj}(K_r,d)$ has dimension vector $\dimu(V_1,V_2)$, we conclude $\rep_{\proj}(K_r,d) \cap \rep(K_r;V_1,V_2) \neq \emptyset$.

Now we consider the case $\Delta_{(V_1,V_2)}(d) \geq d(r-d)$. By the first case, we may assume that $\dim_{\bmk} V_2 < r \dim_{\bmk} V_1$ holds. We consider $(r \dim_{\bmk} V_1-\dim_{\bmk} V_2,\dim_{\bmk} V_1) \in \NN_0^2 \setminus \{(0,0)\}$. Then $\nabla_{(r \dim_{\bmk} V_1-\dim_{\bmk} V_2,\dim_{\bmk} V_1)}(r-d) = (r-d) \dim_{\bmk} V_1 - (r \dim_{\bmk} V_1-\dim_{\bmk} V_2) = \Delta_{(V_1,V_2)}(d) \geq d(r-d)$.
We apply \cref{2.11} and find $M \in \rep_{\esp}(K_r,r-d)$ with dimension vector $(r \dim_{\bmk} V_1-\dim_{\bmk} V_2,\dim_{\bmk} V_1)$. Since $\rep_{\esp}(K_r,r-d)$ does not contain $I_0 = S_1$, we conclude $\dimu \sigma_{K_r}^{-1}(M) = \dimu(V_1,V_2)$ and \cref{2.5} implies $\sigma_{K_r}^{-1}(M) \in \rep_{\proj}(K_r,d)$. 
\item This follows from \cref{2.2}(1) in conjunction with (2).
\end{enumerate}
\end{proof}

We record two more consequences that we will need in the next section for the study of elementary representations. 

\begin{corollary}\label{2.14}
 Let $V_1,V_2$ be vector spaces such that $V_1 \oplus V_2 \neq 0$. The set $\rep_{\inj}(K_r,d) \cap \rep(K_r;V_1,V_2)$ is open in $\rep(K_r;V_1,V_2)$ and non-empty if $-\nabla_{(V_1,V_2)}(d) \geq d(r-d)$.
\end{corollary}
\begin{proof}
Note that $\rep_{\inj}(K_r,d) \cap \rep(K_r;V_1,V_2)$ is open in  $\rep(K_r;V_1,V_2)$, since the duality $D_{K_r} \colon \rep(K_r) \lra \rep(K_r)$ induces an isomorphism of varieties \[\rep(K_r;V_2,V_1) \lra  \rep(K_r;V_1^\ast,V_2^\ast) \cong \rep(K_r;V_1,V_2)  \]
that takes $\rep_{\proj}(K_r,d) \cap \rep(K_r;V_2,V_1)$ to $\rep_{\inj}(K_r,d) \cap \rep(K_r;V_1,V_2)$. If $-\nabla_{(V_1,V_2)}(d) \geq d(r-d)$, we 
$\Delta_{(V_2,V_1)}(d) = \dim_\bmk V_1 - d \dim_\bmk V_2 =  -\nabla_{(V_1,V_2)}(d) \geq d(r-d)$. Hence \cref{2.13} implies that $\rep_{\proj}(K_r,d) \cap \rep(K_r;V_2,V_1)$ is non-empty. By duality \[\rep_{\inj}(K_r,d) \cap \rep(K_r;V_1^\ast,V_2^\ast) \cong \rep_{\inj}(K_r,d) \cap \rep(K_r;V_1,V_2)\]
is non-empty.
\end{proof}

\begin{corollary}\label{2.15}
Let $M \in \rep(K_r)$ be a representation with $(1,r-d) \leq \dimu(M_1,M_2)$. We assume that one of the following conditions holds:
\begin{enumerate}
    \item[(i)] $\nabla_M(d) < d(r-d)$, or
    \item[(ii)] $M \not\in \rep_{\esp}(K_r,d)$.
\end{enumerate}
Then there exists a non-preprojective subpresentation $U_{r-d}$ of $M$ with dimension vector $(1,r-d)$. 
\end{corollary}
\begin{proof} In case (i) we conclude with $\nabla_M(d) < d(r-d)$ and \cref{2.11} that $M \not\in \rep_{\esp}(K_r,d)$, since $M_1 \neq \{0\}$. In case (ii) we apply \cref{2.9} and find a subrepresentation $U \subseteq M$  with dimension vector $(1,a)$ for some $a \in \{0,\ldots,r-d\}$. Since $\dim_{\bmk} M_2 \geq r - d$, we can extend $U$ to a subrepresentation $U_{r-d}$ with dimension vector $(1,r-d)$. The only preprojective indecomposable representation $U$ with dimension vector $\dimu U \leq (1,r-d)$ is $U = P_0$ with dimension vector $\dim P_0 = (0,1)$. Hence $U_{r-d}$ is not preprojective.
\end{proof}

\section{Elementary representations}\label{Section:3}

\subsection{General results} Let $Q$ be a connected and wild quiver.

\begin{Definition}
A non-zero regular representation $E \in \rep(Q)$ is called {\it elementary}, provided there is no short exact sequence 
\[ 0 \lra A \lra E \lra B \lra 0\]
with $A,B \in \rep(Q)$ regular and non-zero.
\end{Definition}

By definition the elementary representations are the simple objects in the full subcategory of regular representations and the analogue of quasi-simple regular representations in the context of tame quivers. Elementary representations for wild quivers were first systematically studied in \cite{KL96} and \cite{Luk92}. 
There, the authors showed that, parallel to the tame situation, there exist only finitely many Coxeter-orbits of dimension vectors of elementary representations. A very useful characterization of elementary representations, established more recently in \cite[Appendix A]{Rin16}, is the following:

\begin{proposition}\label{3.1}
Let $E \in \rep(Q)$ be a non-zero regular representation. The following statements are equivalent.
\begin{enumerate}
    \item $E$ is elementary.
    \item Given any subrepresentation $U$ of $E$, $U$ is preprojective or the quotient $E/U$ is preinjective.
\end{enumerate}
\end{proposition}

Now we return to the case $Q = K_r$ for $r \geq 3$. It is well known (see for example \cite[3.4]{Rin76}) that the region 
\[ \cC_r \coloneqq \{(x,y) \in \NN^2 \mid \frac{1}{r-1} x \leq y < (r-1)x\} \]
is a fundamental domain for the action of the Coxeter transformation $\Phi_r = \begin{pmatrix}
r^2-1 & -r \\
r & -1
\end{pmatrix} \in \GL_2(\ZZ)$\footnote{We identify $\Phi_{K_r} \colon \ZZ^2 \lra \ZZ^2$ and the {\it Coxeter-matrix} $\Phi_r$ with its natural action on $\ZZ^2$ by left multiplication.} on the set 
\[ \cR_r \coloneqq \{ (x,y) \in \NN^2 \mid x^2 + y^2 - rxy  < 1 \}\] of dimension vectors of regular representations in $\rep(K_r)$. Ultimately, we are interested in the set
\[ \widetilde{\cE}_r \coloneqq  \{ (x,y) \in \cC_r \mid \exists E \in \rep(K_r) \ \text{elementary}, \dimu E = (x,y) \}.\]
By \cite[Section 2]{Rin16} the set
\[ \cF_r \coloneqq \{(x,y) \in \NN^2 \mid \frac{2}{r} x \leq y \leq x\} \subseteq \mathcal{C}_r \]
is a fundamental domain for the action of the group $G_r \subseteq \GL_2(\ZZ)$ generated by $\sigma_r$ and the twist function $\delta \colon \ZZ^2 \lra \ZZ^2 ; (x,y) \mapsto (y,x)$ on $\cR$. In fact, the statement was only proven for $r = 3$ but the arguments extend to the general case.

\begin{figure}[!ht]
    \centering
\begin{tikzpicture}[scale=.5]
   \tkzInit[xmax=7.5,ymax=7.5,xmin=0,ymin=0]
   \tkzGrid
   \tkzAxeXY
   \node[color=black]  at (3.5,5) {$\cC_3$};
   \node[color=black]  at (6,5) {$\cF_3$};
   \draw[thick] (0,0) -- (7.5,7.5) node[anchor=south west] {}; % two points for drawing 2x+y=2
   \draw[ thick] (0,0) -- (7.5,15/3) node[anchor=south west] {}; % two points for drawing 2x+y=2
    
       \draw[ thick,dashed] (0,0) -- (4,8) node[anchor=south west] {}; % two points for drawing 2x+y=2
   \draw[ thick] (0,0) -- (7.5,7.5/2) node[anchor=south west] {}; % two points for drawing 2x+y=2
  \tkzText[above](0,6.75){}
  \begin{scope}[on background layer]
  
  \draw[fill = gray!20] (7.5,7.5) node[anchor=north]{}
  -- (0,0) node[anchor=north]{}
  -- (7.5,15/3) node[anchor=north]{};

   \end{scope}
  \end{tikzpicture}
  \captionsetup{justification=centering,margin=2cm}
  \caption{Illustration of $\cC_3$ and $\cF_3$.}
\end{figure}

We define
\[ \cE_r \coloneqq \widetilde{\cE}_r \cap \cF_r = \{ (x,y) \in \cF_r \mid \exists E \in \rep(K_r) \ \text{elementary}, \dimu E = (x,y) \}.\]
Given $M \in \rep(K_r)$ regular indecomposable, we have $\dimu \sigma_{K_r}(M) = \sigma_r(\dimu M)$ and $\dimu D_{K_r}(M) = \delta(\dimu M)$. Since $M$ is elementary if and only if its dual (respectively its $\sigma_{K_r}$-shift$)$ is elementary and $\sigma_r \circ \sigma_r = \Phi_r$, the determination of $\cE_r$ only necessitates the knowledge of $\widetilde{\cE}_r$. The set $\cE_3$ has been determined in \cite{Rin16} and is given by
\[ \cE_3 = \{ (1,1),(2,2)\}.\]
In the following we determine the set $\cE_r$ for arbitrary $r \geq 3$. We start our considerations with the following simple observations, that will be needed later on.

\begin{Lemma}\label{3.2}
Let $M\in \rep(K_r)$ and $\dim_\bmk M_2 \leq 2(r-1)$.
\begin{enumerate}
    \item If $M$ is preinjective, then $M \in \add(I_0 \oplus I_1 \oplus I_2)$.
    \item If $\dimu M \in \cF_r$ and $U \subseteq M$ such $M/U$ is preinjective, then $M/U \in \add(I_0 \oplus I_1)$ and $M/U \cong -\nabla_{M/U}(r) I_0 \oplus \dim_\bmk (M/U)_1 I_1$.
    % \item If $(x,y) \in \cF_r$ and $1 \leq d \leq r - 1$ such that $x - 1 \geq r(y-d)$, then $y \leq 2d$.
    \item If $\dimu M \in \cF_r$ and $U \subseteq M$ such that
    \[ r (\dim_\bmk (M/U)_2) > \dim_\bmk (M/U)_1,\]
    then $M/U$ is not preinjective.
    \item A representation $N \in \rep(K_r)$ with $\dim_\bmk N_1 < L_r\dim_\bmk N_2$ is not preinjective.
    \item A representation $N \in \rep(K_r)$ with $2 \leq \dim_\bmk N_2$ and $\dim_\bmk N_1 < 2r$ is not preinjective.
    % \item If $(x,y) \in \cF_r$ and $1 \leq d \leq r - 1$ such that $x - 1 \geq r(y-d)$, then $y \leq 2d$.
    \end{enumerate}
\end{Lemma}
\begin{proof}
\begin{enumerate}
    \item We have $\dim_\bmk (I_l)_2 \geq \dim_\bmk (I_3)_2 = r^2 - 1$ for all $l \geq 3$. Moreover, we have $r^2-1  > 2(r-1) \geq \dim_\bmk M_2$ since $r \geq 2$.
   \item We apply (1) to $M/U$ and know that $M/U \in \add(I_0 \oplus I_1 \oplus I_2)$. Moreover, we have
   \[ \dim_\bmk (M/U)_1 \leq \dim_\bmk M_1 \leq \frac{r}{2} \dim_\bmk M_2 \leq \frac{r}{2} 2(r-1) = r^2 - r < r^2 - 1 = \dim_\bmk (I_2)_1.\]
   Hence $M/U \in \add(I_0 \oplus I_1)$ and therefore $M/U \cong -\nabla_{M/U}(r) I_0 \oplus \dim_\bmk (M/U)_2 I_1$.
   \item This is a direct consequence of (2) since $\dimu I_1 = (r,1)$ and $\dimu I_0 = (1,0)$.
   \item This follows immediately by applying \cref{1.2} to the direct summands of $N$.
   \item We have $\dim_\bmk (I_l)_1 \geq \dim_\bmk (I_2)_1 = r^2 - 1 > 2r$ for all $l \geq 2$. 
   Assume that $N$ is preinjective. Then $N \in \add(I_0 \oplus I_1)$. Since $\dimu I_0 = (1,0)$ and $\dimu I_1 = (r,1)$, we conclude with $\dim_\bmk N_2 \geq 2$ that $\dim_\bmk N_1 \geq 2r$, a contradiction.
\end{enumerate}    
\end{proof}

\subsection{Restricting y}
\begin{proposition}\label{3.3}
Let $(x,y) \in \cF_r$ with $y \geq r$ and $E$ be a representation with dimension vector $\dimu E = (x,y)$. Then $E$ is not elementary.
\end{proposition}
\begin{proof} Since elementary representations are bricks (see \cite[1.4]{KL96}), we can assume that $E$ is a regular indecomposable representation. From now on we proceed in steps. 
Since $\nabla_E(1) = \dim_\bmk E_2 - \dim_\bmk E_1 \leq 0 < 1(r-1)$
and $\dim_\bmk E_2 = y \geq r - 1$,
we can apply \cref{2.15} and find a non-preprojective subrepresentation $U_{r-1} \subseteq E$ with dimension vector $(1,r-1)$.

At first we assume that $y \geq 2(r-1)$. Then quotient $E/U_{r-1}$ has dimension vector $(x-1,y-r+1)$. We claim that this dimension vector can not belong to a preinjective representation. Indeed, since $r-1 < L_r$, we have $r- 2 L_r < 0$ and conclude 
\begin{align*}
    (x-1)-(y-r+1)L_r &\stackrel{x \leq \frac{r}{2}y}{\leq} (\frac{r}{2}-L_r)y + L_r(r-1)-1 =  (\frac{r-2L_r}{2})y + L_r(r-1)-1 \\
    &\stackrel{y \geq 2(r-1)}{\leq} (\frac{r-2L_r}{2})2(r-1) + L_r(r-1)-1 \\
    &= (r-2L_r)(r-1) + L_r(r-1) - 1  = (r-1)(r-L_r)-1.
\end{align*}  
Recall that $L_r$ is a root of the polynomial $f = X^2 - r X + 1 \in \RR[X]$. Hence \[  (x-1)-(y-r+1)L_r  \leq (r-L_r)(r-1) -  1 = (r-L_r)(r-1) + L_r(L_r - r) = (r- L_r)(r-1 - L_r) < 0, \]
since $r - 1 < L_r < r$. Now \cref{3.2}(4) implies that $E/U_{r-1}$ is not preinjective. We conclude with \cref{3.1} that $E$ is not elementary.

Therefore we can assume from now on that $r \leq y < 2(r-1)$. Given $d \in \NN$ we define 
\[\nabla(d) := r(r-d).\] We begin with the case $\nabla(1) \leq ry-x$. We have $\dimu E/U_{r-1} = (x-1,y-(r-1))$ and  $y-(r-1) \neq 0$. Therefore 
\[ r\dim_\bmk (E/U_{r-1})_2 = r(y-(r-1)) = ry- \nabla(1) \geq x > x - 1 = \dim_\bmk (E/U_{r-1})_1.\]
Since $\dim_\bmk E_2 = y < 2(r-1)$, we can apply \cref{3.2}(3) and conclude that $E/U_{r-1}$ is not preinjective. Now \cref{3.1} implies that $E$ is not elementary.

Now we assume that $ry- x < \nabla(1)$. Since $ry- x \geq r\frac{2}{r}x -x = x \geq y \geq r = \nabla(r-1)$, we find a natural number $2 \leq d \leq r-1$ such that 
\[ \nabla(d) \leq ry - x < \nabla(d-1).\]
We consider two cases:
\begin{itemize}
    \item $d \in \{2,\ldots,r-2\}$, then  $d-(r-1) \leq -1$ and 
\begin{align*}
    \nabla_{(x,y)}(d) - d(r-d) &= \nabla_{(x,y)}(r) -  y(r-d) - d(r-d) < \nabla(d-1) - (y+d)(r-d)\\
    &\leq \nabla(d-1) - (r+d)(r-d) = r(r-d+1)-r^2+d^2 \\
    &= -rd+d^2+r = r + d(d-r) \\
    &= (d-1)(d-(r-1))+1 \leq (d-1) \cdot (-1) + 1 \leq -1 + 1 = 0,
\end{align*}    
By \cref{2.15} there is a subrepresentation $U_{r-d} \subseteq E$ with dimension vector $(1,r-d)$ that is not preprojective. We have $\dimu E/U_{r-d} = (x-1,y-(r-d))$ and the choice of $d$ gives us
\[ r\dim_\bmk (E/U_{r-d})_2 = r(y-(r-d)) = ry - \nabla(d) \geq x > x - 1 = \dim_{\bmk} (E/U_{r-d})_1.\]
Therefore $E/U_{r-d}$ is not preinjective by \cref{3.2}(3) and \cref{3.1} implies that $E$ is not elementary.
\item $d = r - 1$, i.e. $\nabla(r-1) \leq ry- x < \nabla(r-2)$. We get
\[ x > r(y-2) \geq r(\frac{2}{r}x-2) = 2x-2r \Leftrightarrow x < 2r\]
and conclude with $r(y-2) < x < 2r$ that $y < 4$. Since $3 \leq r \leq y < 4$,  we conclude $r = 3$. Hence the statement follows since $\cE_r = \{(1,1),(2,2)\}$ by \cite{Rin16}.
\end{itemize}
\end{proof}

\subsection{Existence of elementary representations}

For $x,y \in \NN_0$, we define
\[ \cE(x,y) \coloneqq \{ g \in \rep(K_r;\bmk^x,\bmk^y) \mid (\bmk^x,\bmk^y,g) \ \text{elementary}\},\]
and
\[ \cB(x,y) \coloneqq   \{ g \in \rep(K_r;\bmk^x,\bmk^y) \mid (\bmk^x,\bmk^y,g) \ \text{brick}\}.\]
Since elementary representations are bricks (see \cite[1.4]{KL96}) we have $\cE(x,y) \subseteq \cB(x,y)$.

We assume from now on that $(x,y) \in \cF_r$. We recall from \cite[1.2.2]{BF24} that $\cB(x,y)$ is a dense subset of $\rep(K_r;\bmk^x,\bmk^y)$ since $q_r(x,y) \leq 1$. \\
In following we determine under which assumptions on $(x,y)$ the set $\cE(x,y)$ is non-empty. Since \cref{3.3} implies that $\cE(x,y) \neq \emptyset$ can only happen for $y < r$, we assume from now on that $y < r$. Before we tackle the general case, we consider an example that illustrates the strategy of proof.

\begin{example}
We have $(6,3) \in \cF_4$ with $3 < 4 = r$ and claim that $(6,3) \in \cE_4$. We have $\nabla_{(6,3)}(3) = 3  \geq 3 = 3(4-3)$. Moreover, we have 
$-\nabla_{(6,3)}(1) = 3 \geq 1(4-1)$. Hence \cref{2.11} and \cref{2.14} imply that $\rep_{\esp}(K_r,3) \cap \rep_{\inj}(K_r,1) \cap \cB(6,3)$ is non-empty. We fix a representation $E$ in the above set. Let $0 \neq U \subseteq E$ a non-preprojective representation. We now show that $E/U$ is preinjective. Since $U$ is not projective, we find $0 \neq u \in U_1$. We consider the subrepresentation $\langle u \rangle$ generated by $u$. Then $\dimu \langle u \rangle = (1,z)$ for some $z \in \{0,1,2,3\}$. Since $\rep_{\esp}(K_r,3)$ is closed under subrepresentations, we have $\langle u \rangle \in \rep_{\esp}(K_r,3)$ and conclude with \cref{2.11} that $3z-1 = \nabla_{(1,z)}(3) \geq 3(4-3) = 3$. Hence $z \geq 2$. Therefore $\dimu E/\langle u \rangle = (5,b)$ with $b \in \{0,1\}$. Since $E \in \rep_{\inj}(K_r,1)$ and $\rep_{\inj}(K_r,1)$ is closed under images (since $\rep_{\proj}(K_r,1)$ is closed under subrepresentation), we have $E/\langle u \rangle  \in \rep_{\inj}(K_r,1)$. Now we apply \cref{2.4} to conclude that $E/\langle u \rangle$ is preinjective. Finally, the presence of the canonical epimorphism $E/\langle u \rangle \lra E/U$ implies that $E/U$ is injective.
\end{example}

Now we consider the general case and start with the following simple observation.
 
\begin{Lemma}\label{3.4}
Let $y = 1$, then $\cE(x,y) = \cB{(x,y)} \neq \emptyset$. 
\end{Lemma}
\begin{proof} Let $M \in \cB(x,y)$, then $M$ is indecomposable and regular. Let $U \subseteq M$ be a proper subrepresentation. Then $0 \neq \dim_\bmk U_2$ and therefore $M/U \in \add(I_0)$ is injective. In particular, $M$ is elementary. This shows $\emptyset \neq \cB(x,y) = \cE(x,y)$.
\end{proof}

We assume from now on that $1 < y < r$ and set $b \coloneqq \lceil \frac{x}{r} \rceil \in \NN$ which is the uniquely determined natural number such that
\[ (b-1)r < x \leq br.\] 

\begin{Remarks} \label{3.5}\phantom{.}
\begin{enumerate}
\item We have $1 \leq b < y < r$: Assume that $\lceil \frac{x}{r} \rceil = b \geq y$. Then $\frac{x}{r} > y  - 1$ and therefore
\[ \frac{r}{2}y \geq x = \frac{x}{r} r > ry - r .\]
Hence $2 > y$, a contradiction since we assume $2 \leq y$.
\item We extend to definition of $\rep_{\inj}(K_r,d)$ to $d \in \{0,\ldots,r-1\}$ be setting $\rep_{\inj}(K_r,0) \coloneqq \rep(K_r)$.
\end{enumerate}
\end{Remarks}

\begin{proposition}\label{3.6}
Let $(x,y) \in \cF_r$ with $1 < y < r$ and $b \coloneqq \lceil \frac{x}{r} \rceil$. The following statements hold.
\begin{enumerate}
    \item If $\cE(x,y)$ is non-empty, then 
    \[(b-1)(y + r - (b-1)) \leq x \leq b(r-y+b)\]
    and $\cE(x,y) \subseteq \cB{(x,y)} \cap \rep_{\esp}(K_r,r-y+b) \cap \rep_{\inj}(K_r,b-1)$.
    \item If \[(b-1)(y + r - (b-1)) \leq x \leq b(r-y+b),\]
    then $\cE(x,y)$ is a non-empty open set given by
    \[ \cE(x,y) = \cB{(x,y)} \cap \rep_{\esp}(K_r,r-y+b) \cap \rep_{\inj}(K_r,b-1).\] 
\end{enumerate}
\end{proposition}
\begin{proof}
\begin{enumerate}
    \item Let $E$ be an elementary representation with dimension vector $\dimu E = (x,y)$. We denote by $F \coloneqq D_{K_r}(E)$ the dual representation with dimension vector $(y,x)$. We proceed in steps.
\begin{enumerate}
    \item[(i)] We have $(b-1)(y + r - (b-1)) \leq x$ and $E  \in \rep_{\inj}(K_r,b-1)$: We assume that $x < (b-1)(y + r - (b-1))$ or $E \not\in \rep_{\inj}(K_r,b-1)$. In both cases we conclude $b \neq 1$ and therefore $b-1 \in \{1,\ldots,r-1\}$. If $x < (b-1)(y + r - (b-1))$, we have \[\Delta_{F}(b-1) = x - (b-1)y < (b-1)(r - (b-1)) \ \text{and} \ x \leq \frac{r}{2}y < ry \ \text{gives}\]
  \begin{align*}
      \Delta_F(b-1) &= x - (b-1)y < ry -  (b-1)y = y(r-(b-1)) \\
      &=\dim_\bmk F_1 (r-(b-1)).
  \end{align*} 
Hence \cref{2.2}(3) implies $F\not\in \rep_{\proj}(K_r,b-1)$. If $E \not \in \rep_{\inj}(K_r,b-1)$, we immediately get $F \not \in \rep_{\proj}(K_r,b-1)$ from the definition.\\

\textbf{The book-keeping}: In both cases we arrive at $F \not \in \rep_{\proj}(K_r,b-1)$ with $b-1 \neq 0$.\\

In view of \cref{2.9} we find $a \in \{1,\ldots,b-1\}$ and subrepresentation $Y \not\in \rep_{\proj}(K_r,b-1)$ of $F$ with $\dimu Y =  (a,a')$ and $a' \leq ar-1 \leq  (b-1)r-1$.  The inequality $(b-1)r-1 < x$ ensures that we can extend $Y$ with a semisimple projective direct summand to a subrepresentation $Y$ of $F$ with dimension vector $(a,(b-1)r-1)$ that satisfies $Y \not\in \rep_{\proj}(K_r,b-1)$. In particular, $Y$ is not preprojective by \cref{2.9}.

Since $F$ is elementary, we can apply \cref{3.1} to conclude that $(y-a,x-(b-1)r+1) = \dimu F/X$ belongs to a preinjective representation. But this is impossible since $x-(b-1)r+1 \geq 1$, $y-a < r$ and $\dim_\bmk (I_l)_1 \geq r$ for all $l \geq 1$ and $\dimu I_0 = (1,0)$.
Hence $(b-1)(y + r - (b-1)) \leq x$ and $E \in \rep_{\inj}(K_r,b-1)$.
\item[(ii)] We have $x \leq b(r-y+b)$ and $E \in \rep_{\esp}(K_r,r-y+b)$:
We assume that $x > b(r-y+b)$. We set $d \coloneqq r - (y-b)$ and note that $d \in \{1,\ldots,r-1\}$ by \cref{3.5}. We get 
\begin{align*}
    \nabla_E(d) - d(r-d) &= d(y-r+d)-x \\
    &= b(r-y+b) - x < 0.
\end{align*} 
Hence $E \not\in \rep_{\esp}(K_r,d)$ by \cref{2.11}. Since $r - d = y - b \leq y$, we conclude with \cref{2.15} that there exists a non-preprojective subrepresentation $U_{r-d} \subseteq E$ with dimension vector $(1,r-d)$. Once again we apply \cref{3.1} and conclude that $E/U_{r-d}$ with dimension vector $(x-1,y-(r-d)) = (x-1,b)$ is preinjective. We apply \cref{3.2}(3) and conclude  $br \leq x - 1$. But this is a contradiction to the definition of $b$ since $x \leq br$.\\
We note that this also shows $E \in \rep_{\esp}(K_r,d) =  \rep_{\esp}(K_r,r-y+b)$. 
\end{enumerate}
\item We set $d \coloneqq  r - (y-b) \in \{1,\ldots,r-1\}$ and have 
\[ \nabla_E(d) - d(r-d) = d(y-r+d) - x = b(r-y+b) - x \geq 0,\]
and 
\[ \quad -\nabla_E(b-1) = x - (b-1)y \geq  (b-1)(r-(b-1)).\]
We can apply \cref{2.11} and \cref{2.14} (for $b \neq 1$) to conclude that $\cB{(x,y)} \cap \rep_{\esp}(K_r,d) \cap \rep_{\inj}(K_r,b-1)$ is non-empty (for $b = 1$ we have $\rep_{\inj}(K_r,b-1) = \rep(K_r)$). We fix a representation $E$ in this space and show now that $E$ is elementary.\\
Let $U \subseteq E$ be a non-preprojective representation, then we find $u \in U \setminus \{0\}$. Recall from \cref{Section2.2} that $\rep_{\esp}(K_r,d)$ is closed under subrepresentation. Therefore the subrepresentation $\langle u \rangle$ generated by $u$ is in $\rep_{\esp}(K_r,d)$ and $\dimu \langle u \rangle = (1,z)$ for some $z \in \{1,\ldots,y\}$. We conclude with \cref{2.11} that 
\[ dz - 1  = \nabla_{\langle u \rangle}(d) \geq d(r-d) \Leftrightarrow d(z-(r-d)) \geq 1.\]
In particular, $z \geq r-d +1 = y - (b-1)$. In other words, $E/\langle u \rangle$ satisfies $\dimu E/\langle u \rangle = (x-1,a)$ with $0 \leq a \leq b-1$. 
If $b = 1$, we conclude that $a = 0$ and therefore $\dimu E/\langle u \rangle$ is injective and the the presence of the canonical epimorphism $E/\langle u \rangle \lra E/U$ implies that $E/U$ is injective.
If $b \neq 1$ we have $E \in \rep_{\inj}(K_r,b-1)$ with $b-1 \neq 0$. Since $\rep_{\inj}(K_r,b-1)$ is closed under images (since $\rep_{\proj}(K_r,b-1)$ is closed under subrepresentation), we have $\dimu E/\langle u \rangle  \in \rep_{\inj}(K_r,b-1)$ and can apply \cref{2.4} to conclude that $E/\langle u \rangle$ is injective and presence of the canonical epimorphism $E/\langle u \rangle \to E/U$
implies that $E/U$ is preinjective. 
Hence $E$ is elementary by \cref{3.1}.\\

\noindent \textbf{The book-keeping}: We have shown that
\[ \emptyset \neq \cB{(x,y)} \cap \rep_{\esp}(K_r,r-y+b) \cap \rep_{\inj}(K_r,b-1) \subseteq \cE(x,y). \]
Now we apply (1) to finish the proof.
\end{enumerate} 
\end{proof}

\begin{Remark}\label{3.7} We extend the definition of $\rep_{\esp}(K_r,d)$ to $\{1,\ldots,r\}$ by setting $\rep_{\esp}(K_r,r) \coloneqq \rep(K_r)$.
\end{Remark}

\begin{Theorem}\label{3.8}
Let $(x,y) \in \cF_r$.
\begin{enumerate}
    \item $\cE(x,y) \neq \emptyset$ implies $y < r$.
    \item For $y < r$ the following statements are equivalent.
    \begin{enumerate}
    \item[(i)] $\cE(x,y) \neq \emptyset$.
    \item[(ii)] $(\lceil \frac{x}{r} \rceil -1)(y+r-(\lceil\frac{x}{r} \rceil-1)) \leq x \leq \lceil\frac{x}{r}\rceil(r-y+\lceil\frac{x}{r}\rceil)$.
\end{enumerate}
If one the equivalent statements holds, we have
\[\cE(x,y) =   \cB{(x,y)} \cap \rep_{\esp}(K_r,r-y+\lceil \frac{x}{r} \rceil) \cap \rep_{\inj}(K_r,\lceil \frac{x}{r} \rceil -1).\] 
\end{enumerate}
\end{Theorem}
\begin{proof}
\begin{enumerate}
\item This is the statement of \cref{3.3}.
\item At first we assume that $y = 1$. Then $x \leq \frac{r}{2}y < r$ and $\lceil \frac{x}{r} \rceil = 1$. So in this case the inequalities in (ii) are always satisfied and by \cref{3.4} we have $\cE(x,y) = \cB(x,y) \neq \emptyset$ as well as
\begin{align*}
     \cE(x,y) &= \cB(x,y) = \cB(x,y) \cap \rep_{\esp}(K_r,r) \cap \rep_{\inj}(K_r,0) \\ 
     &=\cB(x,y) \cap  \rep_{\esp}(K_r,r-y+\lceil \frac{x}{r} \rceil) \cap \rep_{\inj}(K_r,\lceil \frac{x}{r} \rceil - 1).
\end{align*}
Now we assume that $1 < y < r$. Then the equivalence of (i) and (ii) is precisely the statement of \cref{3.6}.
\end{enumerate}
\end{proof}

\begin{corollary}\label{3.9}
Let $(x,y) \in \cF_r$ such that $y \leq x < r$. The following statements are equivalent:
\begin{enumerate}
    \item $\cE(x,y) \neq \emptyset$.
    \item $x +y \leq r + 1$.
\end{enumerate}
In this case we have
\[\cE(x,y) =   \cB{(x,y)} \cap \rep_{\esp}(K_r,r-y+1).\] 
\end{corollary}
\begin{proof}
 
We have $\lceil \frac{x}{r} \rceil = 1$. Hence $(x,y) \in \cF$ satisfies the inequality of the above Theorem if and only if $x \leq r - y + 1$. Moreover, we have in this case
\[ \cE(x,y) = \cB{(x,y)} \cap \rep_{\esp}(K_r,r-y+1) \cap \rep_{\inj}(K_r,0) = \cB{(x,y)} \cap \rep_{\esp}(K_r,r-y+1).\]
\end{proof}

\begin{corollary}\label{3.10}
Let $(x,y) \in \cF_r$. The following statements are equivalent.
\begin{enumerate}
    \item $\cE(x,y) \neq \emptyset$.
    \item $\lfloor \frac{x}{r} \rfloor (y+r-\lfloor \frac{x}{r} \rfloor) \leq x \leq \lceil\frac{x}{r}\rceil(r-y+\lceil\frac{x}{r}\rceil)$ and $y < r$.
    \item $ y  \leq \min \{  \lfloor \frac{x}{r} \rfloor+\frac{x}{\lfloor \frac{x}{r} \rfloor}  -r, \lceil \frac{x}{r} \rceil   -\frac{x}{\lceil \frac{x}{r} \rceil} +r,r-1\}$, where we interpret $\lfloor \frac{x}{r} \rfloor+\frac{x}{\lfloor \frac{x}{r} \rfloor}  -r$ as $\infty$ for $1 \leq x < r$.
    \end{enumerate}
    If one of the equivalent statements holds, we have
\[\cE(x,y) =   \cB{(x,y)} \cap \rep_{\esp}(K_r,r-y+\lceil \frac{x}{r} \rceil) \cap \rep_{\inj}(K_r,\lfloor \frac{x}{r} \rfloor).\] 
\end{corollary}
\begin{proof} Assume that $\frac{x}{r} \in \NN$. In this case we have $\frac{x}{r} = \lceil \frac{x}{r} \rceil$ and 
\[ x < ry\Leftrightarrow x > x - \frac{x}{r}y + (\frac{x}{r})^2 \Leftrightarrow x > \lceil\frac{x}{r}\rceil(r-y+\lceil\frac{x}{r}\rceil).\]
Now \cref{3.8} implies $\cE(x,y) = \emptyset$. Hence we can assume $\frac{x}{r} \not\in \NN$. Then $\lceil\frac{x}{r}\rceil -  1 = \lfloor \frac{x}{r} \rfloor$ and \cref{3.8} implies the equivalence of (1) and (2).\\
The equivalence of (2) and (3) follows from direct computation and \cref{3.9}.
\end{proof}

\begin{proposition}
We have
\[\cE_r = \{ (x,y) \in \NN_{\leq \frac{r(r-1)}{2}}\times \NN_{\leq r-1} \mid \frac{2x}{r} \leq y \leq \min \{  \lfloor \frac{x}{r} \rfloor+\frac{x}{\lfloor \frac{x}{r} \rfloor}  -r, \lceil \frac{x}{r} \rceil   -\frac{x}{\lceil \frac{x}{r} \rceil} +r,x\}\}.\]
\end{proposition}
\begin{proof}
Recall that $(x,y) \in \cF_r$ with $\cE(x,y)$ implies $\frac{2x}{r} \leq y \leq x$ and $y \leq r - 1$. In particular, $x \leq \frac{r(r-1)}{2}$.
\end{proof}

\begin{examples} In the following we discuss the cases $r = 3,4$ in detail to illustrate how to apply our formulas.
\begin{enumerate}

\item The case $r = 3$. We have $1 \leq x \leq \frac{r(r-1)}{2} = 3$ and $1 \leq y \leq r-1 = 2$. We consider the inequalties
\begin{align*}
    &\underline{x = 1}:  \frac{2}{3} \leq y \leq \min \{\infty,3,x = 1, r-1=2\} = 1,
    &\underline{x = 2}:  \frac{4}{3} \leq y \leq \min \{\infty,2,2,2\} = 2, \\
    &\underline{x = 3}:  2 \leq y \leq \min \{1,1,3,2\} = 1.
\end{align*}
This shows $\cE_3 = \{(1,1),(2,2)\}$. Moreover, we have
\[ \cE(1,1) = \cB{(1,1)} \cap \rep_{\esp}(K_3,3) \cap \rep_{\inj}(K_3,0) = \cB{(1,1)} \ \text{and}\] 
\[ \cE(2,2) = \cB{(2,2)} \cap \rep_{\esp}(K_3,2) \cap \rep_{\inj}(K_3,0) = \cB{(1,1)} \cap \rep_{\esp}(K_3,2).\] 
The following figure on the left-hand side shows the elementary dimensions vector in $\cE_3$ and the figure on the right hand side shows $\widetilde{\cE}_3$. The dashed red line is the restriction $y \leq r - 1 = 2$.
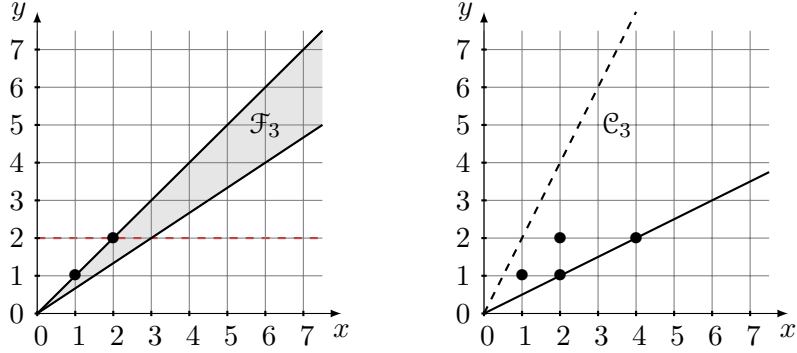
\begin{center}
\begin{minipage}[t]{0.35\textwidth}
\begin{tikzpicture}[scale=.5]
   \tkzInit[xmax=7.5,ymax=7.5,xmin=0,ymin=0]
   \tkzGrid
   \tkzAxeXY
    \node[] at (1,1) {$\bullet$};
   \node[] at (2,2) {$\bullet$};
   \node[color=black]  at (6,5) {$\cF_3$};
   \draw[thick] (0,0) -- (7.5,7.5) node[anchor=south west] {}; % two points for drawing 2x+y=2
   \draw[thick] (0,0) -- (7.5,15/3) node[anchor=south west] {}; % two points for drawing 2x+y=2
  \tkzText[above](0,6.75){}
  \begin{scope}[on background layer]
  
  \draw[fill = gray!20] (7.5,7.5) node[anchor=north]{}
  -- (0,0) node[anchor=north]{}
  -- (7.5,15/3) node[anchor=north]{};
  
    \draw[thick, dashed, red] (0,2) -- (7.5,2) node[anchor=south west] {};
  
   \end{scope}
  \end{tikzpicture}
  \end{minipage}    
  \begin{minipage}[t]{0.35\textwidth}
\begin{tikzpicture}[scale=.5]
   \tkzInit[xmax=7.5,ymax=7.5,xmin=0,ymin=0]
   \tkzGrid
   \tkzAxeXY
   \draw[ thick,dashed] (0,0) -- (4,8) node[anchor=south west] {}; % two points for drawing 2x+y=2
   \draw[ thick] (0,0) -- (7.5,7.5/2) node[anchor=south west] {}; % two points for drawing 2x+y=2
   \node[color=black]  at (3.5,5) {$\cC_3$};
   \node[] at (1,1) {$\bullet$};
   \node[] at (2,1) {$\bullet$};
   \node[] at (2,2) {$\bullet$};
   \node[] at (4,2) {$\bullet$};
  \tkzText[above](0,6.75){}
  \end{tikzpicture}
  \end{minipage}    
  
\end{center}

\item The case $r = 4$.  We have $1 \leq x \leq \frac{r(r-1)}{2} = 6$ and $y \leq r-1 = 3$. We consider the inequalties
 \begin{align*}
    &\underline{x = 1}: \frac{1}{2} \leq y \leq \min \{\infty,4,1,3\} = 1, &\underline{x = 2}: 1 \leq y \leq \min \{\infty,3,2,3\} = 2\\
    &\underline{x = 3}: \frac{3}{2} \leq y \leq \min \{\infty,2,3,3\} = 2,  &\underline{x = 4}: 2 \leq y \leq \min \{1,1,4,3\} = 1\\
    &\underline{x = 5}: \frac{5}{2} \leq y \leq \min \{2,\frac{7}{2},5,3\} = 2,   &\underline{x = 6}: 3 \leq y \leq \min \{3,3,6,3\} = 3.
\end{align*} 
Hence $\cE_4 = \{ (1,1),(2,1),(2,2),(3,2),(6,3)\}$.
Moreover, we have
\begin{align*}
    \cE(1,1) &= \cB{(1,1)} \cap \rep_{\esp}(K_4,4) \cap \rep_{\inj}(K_4,0) = \cB(1,1),\\
    \cE(2,1) &= \cB{(2,1)} \cap \rep_{\esp}(K_4,4) \cap \rep_{\inj}(K_4,0) = \cB(2,1),\\
    \cE(2,2) &= \cB{(2,2)} \cap \rep_{\esp}(K_4,3) \cap \rep_{\inj}(K_4,0) = \cB{(2,2)} \cap \rep_{\esp}(K_4,3),\\
    \cE(3,2) &= \cB{(3,2)} \cap \rep_{\esp}(K_4,3) \cap \rep_{\inj}(K_4,0) = \cB{(3,2)} \cap \rep_{\esp}(K_4,3),\\
    \cE(6,3) &= \cB{(6,3)} \cap \rep_{\esp}(K_4,3) \cap \rep_{\inj}(K_4,1).
\end{align*}
The figure on the left-hand side shows the elementary dimensions vector in $\cE_4$ and the figure on the right hand side shows $\widetilde{\cE}_4$.
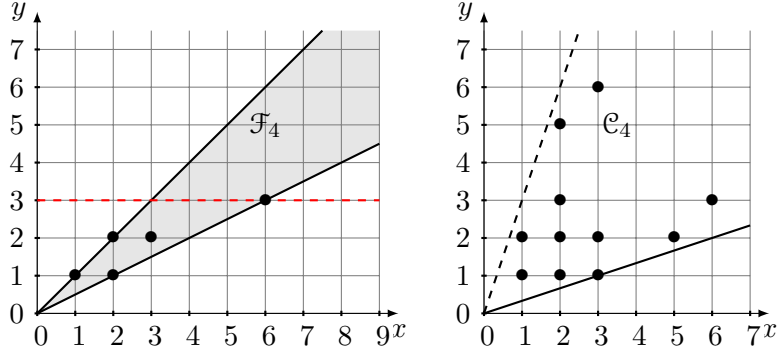
\begin{center}
\begin{minipage}[t]{0.35\textwidth}
    \begin{tikzpicture}[scale=.5]
   \tkzInit[xmax=9,ymax=7.5,xmin=0,ymin=0]
   \tkzGrid
   \tkzAxeXY
   \begin{scope}[on background layer]
    \draw[fill = gray!20,draw=none] (7.5,7.5) node[anchor=north]{}
  -- (0,0) node[anchor=north]{}
  -- (9,4.5) node[anchor=north]{}
  -- (9,7.5) node[anchor=north]{}
  (7.5,7.5) node[anchor=north]{};
 \end{scope}
  
   \draw[ thick] (0,0) -- (7.5,7.5) node[anchor=south west] {}; % two points for drawing 2x+y=2
   \draw[ thick] (0,0) -- (9,4.5) node[anchor=south west] {}; % two points for drawing 2x+y=2
     \draw[thick, dashed, red] (0,3) -- (9,3) node[anchor=south west] {};
   
    \node[color=black]  at (6,5) {$\cF_4$};
   \node[] at (1,1) {$\bullet$};
   \node[] at (2,1) {$\bullet$};

   \node[] at (2,2) {$\bullet$};
   \node[] at (3,2) {$\bullet$};

   \node[] at (6,3) {$\bullet$};

  \tkzText[above](0,6.75){}
  \end{tikzpicture}
  \end{minipage}
  \begin{minipage}[t]{0.35\textwidth}
   \begin{tikzpicture}[scale=.5]
   \tkzInit[xmax=7,ymax=7.5,xmin=0,ymin=0]
   \tkzGrid
   \tkzAxeXY
   \draw[ thick,  dashed] (0,0) -- (2.5,2.5*3) node[anchor=south west] {}; % two points for drawing 2x+y=2
   \draw[ thick] (0,0) -- (7,7/3) node[anchor=south west] {}; % two points for drawing 2x+y=2
   \node[color=black]  at (3.5,5) {$\cC_4$};

   \node[] at (1,1) {$\bullet$};
   \node[] at (2,1) {$\bullet$};
    \node[] at (3,1) {$\bullet$};
    
   \node[] at (1,2) {$\bullet$};
   \node[] at (2,2) {$\bullet$};
   \node[] at (3,2) {$\bullet$};
  \node[] at (5,2) {$\bullet$};
  
 \node[] at (2,3) {$\bullet$};
   \node[] at (6,3) {$\bullet$};
 
   \node[] at (2,5) {$\bullet$};
 
  \node[] at (3,6) {$\bullet$};
    
  \tkzText[above](0,6.75){}
  \end{tikzpicture}
  \end{minipage}
\end{center}

\end{enumerate}
\end{examples}

\section{Orbits of elementary representations}

It has been shown in \cite{Rin16} that elementary representations $E$ with dimension vector in $\cE_3 = \{(1,1),(2,2)\}$ can be described combinatorially in terms of their coefficient quiver. More precisely: There exists a basis $\alpha,\beta,\gamma$ of the arrow space $A_3$ such that the coefficient quiver of $E$ has one of the following two forms:

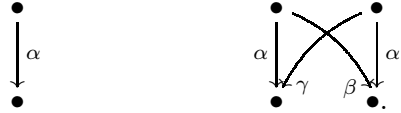
\[
\xymatrix{
\bullet \ar[d]^{\alpha} & & & \bullet \ar[d]_{\alpha} \ar@/^/[dr]_>>{\beta} & \bullet \ar[d]^{\alpha} \ar@/_/[dl]^>>{\gamma} \\
\bullet & & & \bullet & \bullet.
}
\]
In the following we rephrase this result in terms of an algebraic group acting on the variety of representations. Let $V_1,V_2$ be vector spaces. We consider the canonical action of the general linear group $\GL(A_r)$ on $\rep(K_r;V_1,V_2)$: Given $g \in \GL(A_r)$ and $f \in \rep(K_r;V_1,V_2)$, we write $g^{-1}(\gamma_i) = \sum^r_{j=1} \lambda_{ij}^{(g)} \gamma_j$ with $\lambda_{ij}^{(g)} \in \bmk$ for all $i \in \{1,\ldots,r\}$ and let $f^{(g)} \in \rep(K_r;V_1,V_2)$ be the tuple with entries
\[ (f^{(g)})_i = \sum^r_{j=1} \lambda_{ij}^{(g)} f_j, 1 \leq i \leq r.\]
The algebraic group 
\[G_{(V_1,V_2)} \coloneqq \GL(A_r) \times \GL(V_2) \times \GL(V_1)\] 
acts on the space of representations $\rep(K_r;V_1,V_2)$ via
\begin{align*}
    G_{(V_1,V_2)} \times \rep(K_r;V_1,V_2) &\lra \rep(K_r;V_1,V_2) \\
    ((g,h_2,h_1),f) &\mapsto ((h_2 \circ f_i \circ h_1^{-1})_{1 \leq i \leq r})^{(g)} = (h_2 \circ (f^{(g)})_i \circ h_1^{-1})_{1 \leq i \leq r}.
\end{align*} 
Note that $\dim G_{(V_1,V_2)} = r^2 + (\dim_\bmk V_1)^2 + (\dim_\bmk V_2)^2$.
Moreover, we have an action of $\GL(A_r)$ on $\rep(K_r)$ 
\[ \GL(A_r) \times \rep(K_r) \lra \rep(K_r) ; (g,N) \mapsto N^{(g)} \coloneqq (N_1,N_2,(N(\gamma_i))_{1 \leq i \leq r}^{(g)})\]
and an induced action on the isomorphism classes of Kronecker representations $[N]^{(g)} \coloneqq [N^{(g)}]$.
\bigskip
Now let $M_1,M_2$ be vector spaces and $\emptyset \neq \cO \subseteq \rep(K_r;M_1,M_2)$ be a $G_{(M_1,M_2)}$-invariant subset. We let $[\cO] \coloneqq \{ [N] \mid N \in \rep(K_r), \exists f \in \cO : N \cong (M_1,M_2,f) \}$. By definition we have a one-to-one correspondence between $\cO/G_{(M_1,M_2)}$ and $[\cO]/\GL(A_r)$. For $(x,y) \in \NN^2$ we let 
\[ G_{(x,y)} \coloneqq \GL(A_r) \times \GL(\bmk^x) \times \GL(\bmk^y).\]
Since regular representations are $\GL(A_r)$-invariant, the set $\cE(x,y) \subseteq \rep(K_r;\bmk^x,\bmk^y)$ is $G_{(x,y)}$-invariant. Since $\GL(A_r)$ acts transitive on bases of $A_r$, we can rephrase the aforementioned results as follows.

\begin{Theorem}(see \cite[Theorem]{Rin16})\label{4.1} The following statements hold.
\begin{enumerate}
    \item We have $\cE_3 = \{(1,1),(2,2)\}$.
    \item The sets $\cE(1,1)$, $\cE(2,2)$ are orbits under the action of $G_{(1,1)}$ and $G_{(2,2)}$ on $\rep(K_3;\bmk,\bmk)$ and $\rep(K_3;\bmk^2,\bmk^2)$, respectively. 
    \item Let $M \in \rep(K_3)$ be a representation with dimension vector $(1,1)$. The representation is elementary if and only if there is $g \in \GL(A_3)$ such that $M^{(g)} \cong (\bmk,\bmk,(\id_{\bmk},0,0))$.
    \item Let $M \in \rep(K_3)$ be a representation with dimension vector $(2,2)$. The representation is elementary if and only if there is $g \in \GL(A_3)$ such that $M^{(g)} \cong (\bmk^2,\bmk^2,(\id_{\bmk^2}, \beta,\gamma))$ with $\beta(a,b) = (0,a)$ and $\gamma(a,b) = (b,0)$ for all $(a,b) \in \bmk^2$.
\end{enumerate}
\end{Theorem}

In following we show that we can not hope for such a nice classification in case $r \geq 4$.

\begin{Lemma}\label{4.2}
Let $\emptyset \neq \cO \subseteq \rep(K_r;\bmk^x,\bmk^y)$ be a non-empty open and $G_{(x,y)}$-invariant subset of $\rep(K_r;\bmk^x,\bmk^y)$ such that $q_{r}(x,y) < -r^2$. Then $\cO/G_{(x,y)}$ is not finite.
\end{Lemma}
\begin{proof} We set $G \coloneqq G_{(x,y)}$. We assume that $\cO/G$ is finite and fix $T_1,\ldots,T_n \in \cO$ such that $\cO = \bigcup^n_{i=1} G.T_i$. Hence
\[\rep(K_r;\bmk^x,\bmk^y) = \overline{\cO} = \bigcup^n_{i=1} \overline{G.T_i}.\]
Because $\rep(K_r;\bmk^x,\bmk^y)$ is irreducible, we find $i \in \{1,\ldots,n\}$ such that $\rep(K_r;\bmk^x,\bmk^y) =  \overline{G.T_i}$.
Since orbits are open in their closure (\cite[8.3]{Hum75}), we conclude with \cite[1.10]{Har77} that
\[ \dim G \geq \dim G.T_i = \dim \overline{G.T_i} = \dim \rep(K_r;\bmk^x,\bmk^y) = r xy.\]
In particular, we have
\begin{align*}
    0 \leq \dim G - r xy &= r^2 + x^2 + y^2 - r xy = r^2 + q_{r}(x,y),
\end{align*} 
in contradiction to the assumption.
\end{proof}

\begin{corollary}\label{4.3}
Let $(x,y) \in \cE_r$ such that $q_{r}(x,y) < -r^2$. Then $\cE(x,y)/G_{(x,y)}$ is not finite.

\end{corollary}
\begin{proof}
Since $(x,y) \in \cE_r$, we can apply \cref{3.8} and conclude that 
\[ \cE(x,y) =   \cB{(x,y)} \cap \rep_{\esp}(K_r,r-y+\lceil \frac{x}{r} \rceil) \cap \rep_{\inj}(K_r,\lceil \frac{x}{r} \rceil-1) \]
is a non-empty open subset of $\rep(K_r;\bmk^x,\bmk^y)$. Moreover, $\cE(x,y)$ is $G_{(x,y)}$-invariant, since regular representations are $\GL(A_r)$-invariant. Now we apply \cref{4.2}.
\end{proof}

\begin{Theorem}\label{4.4}
Let $r \geq 4$. Then there are infinitely pairwise non-isomorphic elementary representations with dimension vector $(r+2,3)$ that all are in different $\GL(A_r)$-orbits.
\end{Theorem}
\begin{proof} We set $x = r + 2$ and $y = 3 < r$. Then $(x,y) \in \cF_r$ and
\[ \lfloor \frac{x}{r} \rfloor(y+r-\lfloor \frac{x}{r} \rfloor) = r-2 \leq x \leq 2(r-1) = \lceil \frac{x}{r} \rceil (r-y+\lceil \frac{x}{r} \rceil).\]
Now \cref{3.10} implies that $(r+2,3) \in \cE_r$. Moreover, we have $q_{K_r}(r+2,3) = - 2r^2 - 2r + 13 < -r^2$ and can apply \cref{4.3}
\end{proof}

\begin{corollary}\label{4.5}
Let $(x,y) \in \NN^2$ such that $q_{r}(x,y) <  -r^2$. The number of different $\GL(A_r)$-orbits of isomorphism-classes of elementary representation with dimension vector $(x,y)$ is either $0$ or infinite.
\end{corollary}
\begin{proof}
We can assume that there is $E \in \rep(K_r)$ elementary with dimension vector $(x,y)$. By applying $D_{K_r}$ and powers of $\sigma_{K_r}$ to $E$ we find an elementary representation $F$ with dimension vector $\dimu F \in \cE_r$. Since $\sigma_{K_r}$ and $D_{K_r}$ do not change the quadratic form, we have
\[ q_{r}(\dimu F) =  q_{r}(x,y) < -r^2 .\]
Now \cref{4.3} implies that we get infinitely many orbits. Since $D_{K_r}$ and $\sigma_r$ respect $\GL(A_r)$-orbits (see for example \cite[6.1.3]{BF24}), the statements follows.
\end{proof}

\begin{Remark}
Let $E \in \rep(K_3)$ be elementary. Then \cite{Rin16} implies that $q_3(\dimu E) \in \{-1,-4\}$. Hence $q_3(\dimu E) \geq -9 = -r^2$.
\end{Remark}

\section{Examples}

The following figures illustrate our findings for $r \in \{3,4,5,6,7\}$. The dashed red line is the restriction $y \leq r - 1$. We would like to remark that simulations for $ 5\leq r \leq 500$ indicate that a sharp upper bound for $y$ is $\lceil \frac{r}{2} \rceil$.\\

\begin{minipage}[t]{0.3\textwidth}
\centering
\begin{tikzpicture}[scale=.45]
   \tkzInit[xmax=7.5,ymax=7.5,xmin=0,ymin=0]
   \tkzGrid
   \tkzAxeXY
   \node[] at (1,1) {$\bullet$};
   \node[] at (2,2) {$\bullet$};
   \node[color=black]  at (6,5) {$\cF_3$};
   \draw[thick] (0,0) -- (7.5,7.5) node[anchor=south west] {}; % two points for drawing 2x+y=2
   \draw[thick] (0,0) -- (7.5,15/3) node[anchor=south west] {}; % two points for drawing 2x+y=2
   \tkzText[above](0,6.75){}
   \begin{scope}[on background layer]
   \draw[fill = gray!20] (7.5,7.5) node[anchor=north]{}
    -- (0,0) node[anchor=north]{}
    -- (7.5,15/3) node[anchor=north]{};
   \draw[thick, dashed, red] (0,2) -- (7.5,2) node[anchor=south west] {};
   \end{scope}
  \end{tikzpicture}
\end{minipage}
\begin{minipage}[t]{0.3\textwidth}
\begin{tikzpicture}[scale=.45]
   \tkzInit[xmax=9.5,ymax=7.5,xmin=0,ymin=0]
   \tkzGrid
   \tkzAxeXY
   \begin{scope}[on background layer]
   \draw[fill = gray!20,draw=none] (7.5,7.5) node[anchor=north]{}
    -- (0,0) node[anchor=north]{}
    -- (9.5,9.5/2) node[anchor=north]{}
    -- (9.5,7.5) node[anchor=north]{}
    (7.5,7.5) node[anchor=north]{};
    \end{scope}
    \draw[ thick] (0,0) -- (7.5,7.5) node[anchor=south west] {}; % two points for drawing 2x+y=2
    \draw[ thick] (0,0) -- (9.5,9.5/2) node[anchor=south west] {}; % two points for drawing 2x+y=2
    \draw[thick, dashed, red] (0,3) -- (9,3) node[anchor=south west] {};
    \node[color=black]  at (6,5) {$\cF_4$};
    \node[] at (1,1) {$\bullet$};
    \node[] at (2,1) {$\bullet$};
    \node[] at (2,2) {$\bullet$};
    \node[] at (3,2) {$\bullet$};
    \node[] at (6,3) {$\bullet$};
    \tkzText[above](0,6.75){}
\end{tikzpicture}
\end{minipage}
\begin{minipage}[t]{0.35\textwidth}
\begin{tikzpicture}[scale=.45]
   \tkzInit[xmax=11.5,ymax=7.5,xmin=0,ymin=0]
   \tkzGrid
   \tkzAxeXY
   \node[color=black]  at (6,5) {$\cF_5$};
   \begin{scope}[on background layer]
   \draw[fill = gray!20,draw=none] (7.5,7.5) node[anchor=north]{}
    -- (0,0) node[anchor=north]{}
    -- (11.5,11.5*2/5) node[anchor=north]{}
    -- (11.5,7.5) node[anchor=north]{}
    (7.5,7.5) node[anchor=north]{};
   \end{scope}
   \draw[ thick] (0,0) -- (7.5,7.5) node[anchor=south west] {}; % two points for drawing 2x+y=2
   \draw[ thick] (0,0) -- (11.5,11.5*2/5) node[anchor=south west] {}; % two points for drawing 2x+y=
   \draw[thick, dashed, red] (0,4) -- (11,4) node[anchor=south west] {};
   \node[] at (1,1) {$\bullet$};
   \node[] at (2,1) {$\bullet$};
   \node[] at (2,2) {$\bullet$};
   \node[] at (3,2) {$\bullet$};
   \node[] at (4,2) {$\bullet$};
   \node[] at (3,3) {$\bullet$};
   \node[] at (7,3) {$\bullet$};
    \tkzText[above](0,6.75){}
  \end{tikzpicture}
\end{minipage}

\vspace{0.5cm}

\begin{center}
\begin{tikzpicture}[scale=.5]
    \tkzInit[xmax=15.5,ymax=7.5,xmin=0,ymin=0]
    \tkzGrid
    \tkzAxeXY
    \node[color=black]  at (6,4) {$\cF_6$};
    \begin{scope}[on background layer]
        \draw[fill = gray!20,draw=none] (7.5,7.5) node[anchor=north]{}
        -- (0,0) node[anchor=north]{}
        -- (15.5,15.5*2/6) node[anchor=north]{}
        -- (15.5,7.5) node[anchor=north]{}
        (7.5,7.5) node[anchor=north]{};
   \end{scope}
   \draw[ thick] (0,0) -- (7.5,7.5) node[anchor=south west] {}; % two points for drawing 2x+y=2
   \draw[ thick] (0,0) -- (15.5,15.5*2/6) node[anchor=south west] {}; % two points for drawing 2x+y=2
   \draw[thick, dashed, red] (0,5) -- (15,5) node[anchor=south west] {};
   \node[] at (1,1) {$\bullet$};
   \node[] at (2,1) {$\bullet$};
   \node[] at (3,1) {$\bullet$};
   \node[] at (2,2) {$\bullet$};
   \node[] at (3,2) {$\bullet$};
   \node[] at (4,2) {$\bullet$};
   \node[] at (5,2) {$\bullet$};
   \node[] at (3,3) {$\bullet$};
   \node[] at (4,3) {$\bullet$};
   \node[] at (8,3) {$\bullet$};
   \node[] at (9,3) {$\bullet$};
   %\node[] at (8,4) {$\bullet$};  
   \tkzText[above](0,6.75){}
\end{tikzpicture}

\vspace{0.5cm}

\begin{tikzpicture}[scale=.5]
   \tkzInit[xmax=24.5,ymax=7,xmin=0,ymin=0]
   \tkzGrid
   \tkzAxeXY
    \node[color=black]  at (6,5) {$\cF_7$};
  \begin{scope}[on background layer]
    \draw[fill = gray!20,draw=none] (7,7) node[anchor=north]{}
  -- (0,0) node[anchor=north]{}
  -- (24.5,24.5*2/7) node[anchor=north]{}
  -- (7,7) node[anchor=north]{};
 \end{scope}
   
    \draw[ thick] (0,0) -- (7,7) node[anchor=south west] {}; % two points for drawing 2x+y=2
    \draw[ thick] (0,0) -- (24.5,24.5*2/7) node[anchor=south west] {}; % two points for drawing 2x+y=2
   \draw[thick, dashed, red] (0,6) -- (21,6) node[anchor=south west] {};
   \node[] at (1,1) {$\bullet$};
   \node[] at (2,1) {$\bullet$};
   \node[] at (3,1) {$\bullet$};
   \node[] at (2,2) {$\bullet$};
   \node[] at (3,2) {$\bullet$};
   \node[] at (4,2) {$\bullet$};
   \node[] at (5,2) {$\bullet$};
   \node[] at (6,2) {$\bullet$};
   \node[] at (3,3) {$\bullet$};
   \node[] at (4,3) {$\bullet$};
   \node[] at (5,3) {$\bullet$};
   \node[] at (9,3) {$\bullet$};
   \node[] at (10,3) {$\bullet$};
   \node[] at (4,4) {$\bullet$};
   \node[] at (10,4) {$\bullet$};
   \tkzText[above](0,6.75){}
\end{tikzpicture}
\end{center}

\thispagestyle{empty}
\bibliographystyle{alpha}
\bibliography{Bibtex}

\end{document}